\documentclass[10pt,a4paper]{amsart}

\usepackage{amssymb,amsmath,amsthm,amstext,amscd,stmaryrd}
\usepackage{gitinfo} 
\usepackage{datetime}
\usepackage[english]{babel}

\newdateformat{mydate}{\THEYEAR-\twodigit{\THEMONTH}-\twodigit{\THEDAY}}

\title[Higher order approximations for degenerate equations]{Higher order spatial approximations for degenerate parabolic stochastic partial differential equations}

\author[E.J. Hall]{Eric Joseph Hall}
\address{School of Mathematics\\University of Edinburgh\\King's Buildings\\Edinburgh, EH9 3JZ\\UK}
\email{e.hall@ed.ac.uk}

\numberwithin{equation}{section}
\theoremstyle{plain}
\newtheorem{thm}{Theorem}[section]           
\newtheorem{cor}[thm]{Corollary}             
\newtheorem{lem}[thm]{Lemma}                 
\newtheorem{asm}[thm]{Assumption}            

\theoremstyle{definition}

\theoremstyle{remark}
\newtheorem{rmk}[thm]{Remark}                
\newtheorem{eg}{Example}                     

\newcommand{\mexp}{\mathrm{E}}

\begin{document}

\begin{abstract}
We consider an implicit finite difference scheme on uniform grids in time and space for the Cauchy problem for a second order parabolic stochastic partial differential equation where the parabolicity condition is allowed to degenerate. Such equations arise in the nonlinear filtering theory of partially observable diffusion processes. We show that the convergence of the spatial approximation can be accelerated to an arbitrarily high order, under suitable regularity assumptions, by applying an extrapolation technique. 
\end{abstract}


\maketitle

\section{Introduction}\label{>>sec: Introduction}

Motivated by the nonlinear filtering theory, we consider the Cauchy problem for the stochastic partial differential equation (SPDE)
\begin{equation}\label{eqn: zakai equation}
\mathrm{d} u = \left(a^{\alpha\beta}D_{\alpha}D_{\beta} u + f\right) \mathrm{d}t + \sum_{\rho=1}^{d_{1}} \left( b^{\alpha \rho} D_{\alpha} u + g^{\rho} \right) \mathrm{d}w^{\rho}
\end{equation}
with initial condition $u(0) = u_{0}$ where $(w^{\rho})_{\rho=1}^{d_{1}}$ is a $d_{1}$-dimensional Wiener martingale for integer $d_{1} \geq 1$ and the summation convention is used with respect to the repeated indices $\alpha, \beta \in \{0, \dots, d\}$ for integer $d \geq 1$. Here $D_{\alpha} := \partial/\partial x_{\alpha}$, for $\alpha \in \{1, \dots, d\}$, denotes the spatial derivative with respect to the direction $\alpha$ and we let $D_{0}$ be the identity. A special case of this equation, when the free terms $f$ and $g$ vanish, arises as the equation for the unnormalized conditional density of a signal process with respect to an observation process in the nonlinear filtering theory and is often referred to as the Zakai equation (see \cite{Kunita:1981,Pardoux:1979,Zakai:1969}). The behavior of this equation is governed by the quadratic form $$\sum_{\alpha, \beta =1}^{d} A^{\alpha\beta} z_{\alpha}z_{\beta}$$
for $A^{\alpha\beta} := 2 a^{\alpha\beta} - b^{\alpha\rho} b^{\beta\rho}$ and $z \in \mathbf{R}^{d}$. In \cite{KrylovRozovskii:1982}, it is emphasized that in the setting of the nonlinear filtering theory one is only guaranteed the nonnegative definiteness of the matrices $A$, that is, when \eqref{eqn: zakai equation} satisfies a \emph{degenerate} stochastic parabolicity condition (\emph{cf}.\ \cite{KrylovRozovskii:1977,Pardoux:1975th} where the solvability of this equation is studied under the \emph{uniform nondegeneracy} of the matrices $A$). In application, these problems are high dimensional in nature and the solutions are required on-line. Therefore accurate and efficient numerical methods are desired for solving the Cauchy problem for \eqref{eqn: zakai equation} under a degenerate parabolicity assumption.

The present manuscript concerns the accuracy of a space-time, that is, a fully discretized, finite difference scheme on uniform grids in time and space for approximating the solution to the Cauchy problem for \eqref{eqn: zakai equation} under the degenerate parabolicity assumption. In general the rate of convergence of finite difference schemes is known to be unsatisfactory in high dimensional settings. We prove that the rate of convergence of the spatial approximation for our space-time scheme can be accelerated to an arbitrarily high order with respect to the computational effort by applying an extrapolation technique. That is, we show that the rate of the strong convergence of the spatial approximation to the temporal discretization can be accelerated to any order of accuracy if the initial conditions, coefficients, and free terms are sufficiently smooth in space and the matrices $A$ can be decomposed as 
$$A = \sigma \sigma^{T}$$
for matrices $\sigma$ sufficiently smooth in space. While the requirement that the $A$ admit such a decomposition is quite restrictive, this condition is satisfied in the nonlinear filtering problem even in the general case of correlated signal and observation noises when the diffusion coefficients of the signal noise are sufficiently smooth. 

The extrapolation technique that we employ to obtain higher order convergence is often referred to as \emph{Richardson's method}, after L.F.\ Richardson who suggested the method for accelerating the convergence of finite difference approximations for certain partial differential equations (PDEs) (see \cite{Richardson:1911,RichardsonGaunt:1927}). The method relies on the existence of an asymptotic expansion for the error between the approximate and true solutions to a continuous problem in powers of the discretization parameter. Richardson observed that by taking appropriate weighted averages of the approximation at different mesh sizes certain lower order terms in the expansion vanish yielding a higher order rate of convergence. Therefore, it is important to give sufficient conditions under which such expansions exists. We emphasize that not only does the existence of the asymptotic expansion allow us to apply Richardson's method to an arbitrarily high order, but also it allows us to measure the rate of convergence in the supremum norm. Richardson's method has been thoroughly studied in the literature, see for example the book \cite{MarchukShaidurov:1983} which provides a study of Richardson's method for finite difference schemes for deterministic PDEs and the survey articles \cite{Brezinski:2000,Joyce:1971} on convergence acceleration methods. Part I of the book \cite{Sidi:2003} concerns Richardson's method and algorithms for its implementation for PDEs; our results are of a more theoretical nature and much work still needs to be done to implement them.

While finite difference schemes for PDEs and, to a lesser extent, for SPDEs are well studied in the literature (for example, see \cite{Thomee:1990,JentzenKloeden:2009} and the references therein) there are only a few results for degenerate parabolic equations and even fewer results concerning convergence acceleration for degenerate equations. Sharp rates of convergence are given in \cite{DongKrylov:2005} for monotone finite difference schemes for possibly degenerate parabolic and elliptic deterministic PDEs. In \cite{GyongyKrylov:2011} Richardson's method is applied to monotone finite difference schemes for possibly degenerate parabolic deterministic PDEs to accelerate the rate of convergence. Recently, in \cite{Gyongy:2011} a rate of convergence is given for a class of finite difference methods, that approximate in space via finite differences while allowing time to vary continuously, for degenerate parabolic SPDEs and sufficient conditions are given for accelerating the rate of convergence for the approximation in space. 

The current manuscript extends the results of \cite{Gyongy:2011} to a fully discretized scheme. We also mention \cite{Hall:2012}, where results similar to those of the present manuscript are given under the \emph{strong} parabolicity condition. A principle contribution of the present work is to provide estimates in the supremum norm in appropriate spaces for the solutions to the space-time scheme and the discretization in time under the degenerate parabolicity condition. The methods used to provide the requisite estimates in \cite{Gyongy:2011} are not tenable in the discrete time case. Further, we mention that we have chosen to consider here implicit schemes as we believe that these are favored from a practical standpoint and because such schemes are unconditionally stable. We note that, to the author's knowledge, there are no results that give the rate of convergence of the implicit time scheme for SPDEs under a degenerate parabolicity assumption and it will be the subject of a future work to give such a rate of convergence for the implicit time scheme as well as more general methods. 

The paper is outlined as follows. In the next section, we begin by presenting our time scheme and our space-time scheme for approximating the solution to the Cauchy problem for \eqref{eqn: zakai equation} as well as some preliminaries and assumptions. We then state our main results. Theorem \ref{thm: expansion} gives sufficient conditions for the existence of an asymptotic expansion for the error between the space-time approximation and the temporal discretization in powers of the spatial mesh size. Theorem \ref{thm: generalized expansion} gives sufficient conditions for a generalization of Theorem \ref{thm: expansion}, namely, the existence of such an expansion for differences of the solution. Then Theorems \ref{thm: acceleration} and \ref{thm: acceleration for derivatives}, using the aforementioned expansions, give an accelerated rate of convergence for the spatial approximation and for derivatives of the spatial approximation, respectively. The proof of Theorem \ref{thm: generalized expansion}, and hence Theorem \ref{thm: expansion}, is given in Section \ref{sec: Proof of Main Results} after some preliminary estimates are proven in Section \ref{>>sec: Auxilliary Results}.
 
We end this section by introducing some notation that will be used throughout this work. For integer $d \geq 1$, let $\mathbf{R}^{d}$ be the space of Euclidean points $x = (x_{1}, \dots, x_{d})$. We denote the $\sigma$-algebra of Boreal subsets of $\mathbf{R}^{d}$ by $\mathcal{B}(\mathbf{R}^{d})$. Recall that we denote by $D_{\alpha} := \partial/\partial x_{\alpha}$ for $\alpha \in \{1, \dots, d\}$ the spatial derivative with respect to the direction $\alpha$ and let $D_{0}$ be the identity. For an integer $m \geq 0$, we denote by $W_{2}^{m} := W_{2}^{m}(\mathbf{R}^{d})$ the usual Hilbert-Sobolev spaces of function on $\mathbf{R}^{d}$, defined as the closure of $C_{0}^{\infty}(\mathbf{R}^{d})$ functions $\phi: \mathbf{R}^{d} \to \mathbf{R}^{d}$ in the norm
\begin{equation*}
\| \phi \|_{m}^{2} := \sum_{|\rho| \leq m} \int_{\mathbf{R}^{d}} |D^{\rho} \phi(x)|^{2} \, \mathrm{d}x,
\end{equation*}
where $D^{\rho} = D^{\rho_{1}}_{1} \dots D^{\rho_{d}}_{d}$ for a multiindex $\rho = (\rho_{1}, \dots, \rho_{d})$ of length $|\rho| = \rho_{1} + \dots + \rho_{d}$. For an integer $s \geq 0$, we will use the notation $D^{s}\phi$ to denote the collection of all $s$th order spatial derivatives of $\phi$, that is, $D^{s} \phi := \{ D^{\rho}\phi : |\rho| \leq s\}$ for functions $\phi = \phi (x)$ for $x \in \mathbf{R}^{d}$. We note that for $L^{2} := L^{2}(\mathbf{R}^{d}) = W_{2}^{0}$ we will denote the norm by $\| \cdot{} \|_{0}$ and we will use $( \cdot, \cdot)$ to denote the usual inner product in that space. Let $(\Omega, \mathcal{F}, P)$ be a complete probability space and let $\mathcal{F}(t)$, $t \geq 0$, be an increasing family of sub-$\sigma$-algebras of $\mathcal{F}$ such that $\mathcal{F}(0)$ is complete with respect to $(\mathcal{F}, P)$. For a fixed integer $d_{1} \geq 1$ and a constant $T \in (0,\infty)$ let $(w^{\rho})_{\rho=1}^{d_{1}}$ be a given sequence of independent Wiener processes carried by the complete stochastic basis $(\Omega, \mathcal{F}, (\mathcal{F}(t))_{t \geq 0}, P)$. For the fundamentals of the nonlinear filtering theory, we refer the reader to the book \cite{BainCrisan:2009} and for basic notions and results from the theory of SPDEs we refer the reader to \cite{Rozovskii:1990}.

We collect the following notation for our discretizations and differences. For fixed $\tau \in (0,1)$, we partition $[0,T]$ into a fixed integer $n \geq 0$ with mesh size $\tau$ obtaining the time grid 
$$\left\{i\tau; i \in \{0,1,\dots,n\}, \tau n = T\right\}.$$
We define $\phi_{i} := \phi(i\tau)$ for functions $\phi$ depending on $t \in [0,T]$. In particular for $i \in \{1, \dots, n\}$, we will use the notation 
$$\xi^{\rho}_{i} := w^{\rho}_{i} - w^{\rho}_{i-1}$$
to denote the increments of the Weiner process for each $\rho \in \{1, \dots, d_{1}\}$ and $\mathcal{F}_{i} := \mathcal{F}(i\tau)$ to denote the filtration. For $h \in \mathbf{R} \setminus \{0\}$ and a finite subset $\Lambda \subset \mathbf{R}^{d}$ containing the origin we define the space grids 
$$\mathbf{G}_{h} := \left\{ \lambda_{1} h + \dots + \lambda_{p}h; p \in \{1, 2, \dots \}, \lambda_{i} \in \Lambda \cup (-\Lambda)\right\}$$
with mesh size $|h|$. We denote $\Lambda_{0} := \Lambda \setminus \{0\}$. For all $h \in \mathbf{R}\setminus \{0\}$ we define first order and first order symmetric differences by
\begin{gather*}
\delta_{h,\lambda} := \frac{1}{h} \left( T_{h,\lambda} - I \right) \quad \text{and} \quad \delta_{\lambda} = \delta^{h}_{\lambda} := \frac{1}{2} \left( \delta_{h,\lambda} + \delta_{-h,\lambda} \right) = \frac{1}{2h} \left( T_{h,\lambda} - T_{h,-\lambda}\right),
\end{gather*}
respectively, for $\lambda \in \mathbf{R}^{d}\setminus\{0\}$ where for all $h \in \mathbf{R}$ we define the shift operator
\begin{gather*}
T_{h,\lambda} \phi(x) := \phi(x+h\lambda)
\end{gather*}
for functions $\phi$ on $\mathbf{R}^{d}$. We define $\delta_{h,0} := I$ and $\delta_{0} := I$. We also adopt the notation $N = N(\cdot{})$ for a constant $N$ depending only on the parameters given as arguments. For basic notions from numerical analysis we refer the reader to \cite{MortonMayers:1994,RichtmyerMorton:1967}.

\section{Main results}\label{>>sec: Main results}

To accelerate the rate of convergence of the spatial approximation for a space-time finite difference scheme, we will consider the error between the space-time approximation and the temporal discretization, the latter of which is a continuous problem in the spatial variable. Therefore we begin by considering a discretization in time for \eqref{eqn: zakai equation}, namely the implicit Euler method. We then replace the differential operators with difference operators in this temporal discretization, yielding a space-time scheme. We then state our results, the two main results being Theorem \ref{thm: expansion}, which gives sufficient conditions for the existence of the desired asymptotic expansion for the error between the space-time approximation and the temporal discretization in powers of the spatial mesh size, and Theorem \ref{thm: acceleration}, which uses the expansion to obtain an arbitrarily high order of convergence via Richardson's method.

For each fixed $\tau \in (0,1)$, we consider 
\begin{equation}\label{eqn: time scheme}
v_{i}(x) = v_{i-1}(x) + \left( \mathcal{L}_{i} v_{i}(x) + f_{i}(x) \right)\tau + \sum_{\rho=1}^{d_{1}} \left( \mathcal{M}^{\rho}_{i-1} v_{i-1}(x) + g^{\rho}_{i-1}(x) \right) \xi^{\rho}_{i}
\end{equation}
for $\omega \in \Omega$, $x \in \mathbf{R}^{d}$, and $i \in \{1, \dots, n\}$ with initial condition $v_{0}(x) = v_{0}$ where $\mathcal{L}_{i}$ and $\mathcal{M}_{i}^{\rho}$ are second order and first order differential operators given by $\mathcal{L}_{i} \phi := a^{\alpha\beta}_{i}(x) D_{\alpha} D_{\beta} \phi$ and $\mathcal{M}^{\rho}_{i} \phi := b^{\alpha \rho}_{i}(x) D_{\alpha}\phi$, for $\rho \in \{1, \dots, d_{1}\}$, where the summation convention is used with respect to the repeated indices $\alpha,\beta \in \{0,1,\dots,d\}$. We assume that the given $a^{\alpha\beta}_{i} := a^{\alpha\beta}_{i}(x)$ and $b^{\alpha}_{i} := (b^{\alpha}_{i}(x))_{\rho=1}^{d_{1}}$ are real-valued and $\mathbf{R}^{d_{1}}$-valued, respectively, $\mathcal{F}_{i}\otimes \mathcal{B}(\mathbf{R}^{d})$-measurable functions for $\omega \in \Omega$ and $i \in \{0, \dots, n\}$ for all $\alpha,\beta \in \{0, \dots, d\}$. The free terms $f_{i} := f_{i}(x)$ and $g^{\rho}_{i} := g^{\rho}_{i}(x)$, for $\rho \in \{1, \dots, d_{1}\}$, are $\mathcal{F}_{i} \otimes \mathcal{B}(R^{d})$-measurable functions for every $\omega \in \Omega$, $x \in \mathbf{R}^{d}$, and $i \in \{0, \dots, n\}$. The discretization \eqref{eqn: time scheme} represents an implicit Euler method for approximating the solution to \eqref{eqn: zakai equation} in time. Solutions to \eqref{eqn: time scheme} with appropriate initial conditions are understood as sequences of $W^{1}_{2}$-valued random variables satisfying \eqref{eqn: time scheme} in a weak sense in $W^{-1}_{2}$.  

As discussed in the introduction, we consider the following degenerate stochastic parabolicity condition, necessary for the well-posedness of \eqref{eqn: zakai equation} and hence \eqref{eqn: time scheme}. Note that this is a weaker condition than the \emph{strong stochastic parabolicity} condition which assumes the uniform nondegeneracy of the quadratic form (\emph{cf}.\ Assumption 2.2 in \cite{Hall:2012} for example). 

\begin{asm}\label{asm: stochastic parabolicity}
For all $\omega \in \Omega$, $i \in \{1, \dots, n\}$, $x \in \mathbf{R}^{d}$, and $z = (z_{1},\dots, z_{d}) \in \mathbf{R}^{d}$
$$\sum_{\alpha,\beta = 1}^{d} \left(2 a^{\alpha\beta}_{i} - b^{\alpha \rho}_{i} b^{\beta \rho}_{i} \right) z_{\alpha} z_{\beta} \geq 0, $$
that is, the quadratic form is nonnegative definite (positive semi-definite).
\end{asm}

To formulate existence and uniqueness results, as well as estimates, for the solution to \eqref{eqn: time scheme} we also require some smoothness assumptions on the coefficients, the free terms, and the initial conditions. Let $m \geq 0$ be an integer.

\begin{asm}\label{asm: coefficients of the differential operators}
For each $\omega \in \Omega$ and $i \in \{0, \dots, n\}$, the functions $a^{\alpha,\beta}_{i}$ and the functions $a^{0 \alpha}_{i}$, $a^{\alpha 0}_{i}$, and $a^{00}_{i}$ are, respectively, $(m + 1) \vee 2$ times and $m + 1$ times continuously differentiable in $x$ for $\alpha,\beta \in \{1, \dots,d\}$. For each $\omega \in \Omega$ and $i \in \{0, \dots, n\}$, the functions $b^{\alpha}_{i}$ are $m + 2$ times continuously differentiable in $x$ for $\alpha \in \{0, \dots, d\}$. Further, there exist constants $K_{j}$, for $j \in \{0, \dots, m +2 \}$, such that 
\begin{align*}
\left|D^{j} a^{\alpha \beta}_{i} \right| &\leq K_{j} && \text{for $j \leq (m + 1) \vee 2$,}\\
\left|D^{j} a^{\alpha 0}_{i} \right| + \left|D^{j} a^{0\alpha}_{i} \right| + \left|D^{j} a^{00}_{i} \right| &\leq K_{j} && \text{for $j \leq m + 1$, and}\\
\left|D^{j} b^{\alpha}_{i} \right| + \left|D^{j} b^{0}_{i} \right| &\leq K_{j} &&\text{for $j \leq m + 2$}
\end{align*}
for all $\alpha, \beta \in \{1, \dots, d\}$. 
\end{asm}

For integer $l \geq 0$, we define the norm 
\begin{equation*}
\left\llbracket \phi \right\rrbracket_{l}^{2} := \mexp \sum_{i=0}^{n} \tau \left\| \phi_{i} \right\|_{l}^{2}
\end{equation*}
and let $\mathbf{W}^{l}_{2}(\tau)$ be the space of $W^{l}_{2}$-valued $\mathcal{F}_{i}$-measurable processes $\phi$ such that $\left\llbracket \phi \right\rrbracket_{l}^{2} < \infty$. We use the shorthand notation 
$$\left\llbracket g \right\rrbracket_{l}^{2} := \sum_{\rho=1}^{d_{1}} \left\llbracket g^{\rho}\right\rrbracket_{l}^{2}$$ 
for functions $g = (g^{\rho})_{\rho=1}^{d_{1}}$.

\begin{asm}\label{asm: on initial conditions and free terms}
The initial condition $v_{0} \in L^{2}(\Omega, \mathcal{F}_{0}, W^{m+2}_{2})$, the space of $\mathcal{F}_{0}$-measurable $W^{m+2}_{2}$-valued square integrable functions on $\Omega$. The free terms $f$ and $g^{\rho}$, for $\rho \in \{1, \dots, d_{1}\}$, take values in $\mathbf{W}^{m+1}_{2}(\tau)$. Moreover, 
\begin{equation}\label{eqn: bound on initial conditions and free terms}
\mathcal{K}_{m}^{2} := \tau \mexp \left\| v_{0} \right\|_{m+2}^{2} + \left\llbracket f \right\rrbracket_{m+1}^{2} + \left\llbracket g \right\rrbracket_{m+1}^{2} < \infty.
\end{equation}
\end{asm}

\begin{rmk}\label{rmk: on Sobolev's embedding}
For $m > d/2$, we can find a  continuous function of $x$ which is equal to $v_{0}$ almost everywhere for almost all $\omega \in \Omega$, by Sobolev's embedding of $W_{2}^{m} \subset \mathcal{C}_{b}$, the space of bounded continuous functions. Similarly, for each $\omega \in \Omega$ and $i \in \{0, \dots, n\}$ there exist continuous functions of $x$ which coincide with $f_{i}$ and $g^{\rho}_{i}$, for $\rho \in \{1, \dots, d_{1}\}$, for almost every $x \in \mathbf{R}^{d}$. Thus, if Assumption \ref{asm: on initial conditions and free terms} holds with $m > d/2$ we assume that $v_{0}$, $f_{i}$, and $g^{\rho}_{i}$, for $\rho \in \{1, \dots, d_{1}\}$, are continuous in $x$ for all $i \in \{0, \dots, n\}$.
\end{rmk}


For the time scheme \eqref{eqn: time scheme} we give the following solvability theorem along with an estimate. The proof is provided after some preliminaries are presented in the next section. 

\begin{thm}\label{thm: solvability of the time scheme with estimate}
If Assumptions \ref{asm: stochastic parabolicity}, \ref{asm: coefficients of the differential operators}, and \ref{asm: on initial conditions and free terms} hold, then \eqref{eqn: time scheme} admits a unique $W^{m}_{2}$-valued $\mathcal{F}_{i}$-measurable solution $v$. Moreover, 
\begin{equation}\label{eqn: estimate for the time scheme}
\mexp \max_{i\leq n} \left\| v_{i} \right\|_{m}^{2} \leq N \mathcal{K}_{m}^{2}
\end{equation}
holds for a constant $N = N(d,d_{1},m,T,K_{0},\dots,K_{m+2})$.
\end{thm}

Now we wish to approximate \eqref{eqn: time scheme} in space by replacing the differential operators with difference operators. Together with \eqref{eqn: time scheme} we consider, for a finite subset $\Lambda \subset \mathbf{R}^{d}$ containing the origin,
\begin{equation}\label{eqn: space-time scheme}
v^{h}_{i}(x) = v^{h}_{i-1}(x) + \left( L^{h}_{i} v^{h}_{i}(x) + f_{i}(x) \right) \tau + \sum_{\rho=1}^{d_{1}} \left( M^{h,\rho}_{i-1} v^{h}_{i-1}(x) + g^{\rho}_{i-1}(x) \right) \xi^{\rho}_{i}
\end{equation}
for $\omega \in \Omega$, $x \in \mathbf{R}^{d}$, and $i \in \{1, \dots, n\}$ with initial conditions $v^{h}_{0} (x) = v_{0}$. For each $i \in \{0, \dots, n\}$, the $L^{h}_{i}$ and $M^{h,\rho}_{i}$ are given by 
\begin{gather*}
L^{h}_{i} \phi := \sum_{\lambda,\mu \in \Lambda} \mathfrak{a}^{\lambda \mu}_{i}(x) \delta^{h}_{\lambda} \delta^{h}_{\mu}\phi + \sum_{\lambda \in \Lambda_{0}} \left( \mathfrak{p}^{\lambda}_{i}(x) \delta_{h,\lambda} \phi - \mathfrak{q}^{\lambda}_{i}(x) \delta_{-h,\lambda} \phi \right) \\ \intertext{and} M^{h,\rho}_{i} \phi := \sum_{\lambda \in \Lambda} \mathfrak{b}^{\lambda\rho}_{i}(x) \delta^{h}_{\lambda} \phi
\end{gather*}
for $\rho \in \{1, \dots, d_{1}\}$. For all $\lambda, \mu \in \Lambda$, we assume the given $\mathfrak{a}^{\lambda\mu}_{i} := \mathfrak{a}^{\lambda\mu}_{i}(x)$, $\mathfrak{p}^{\lambda}_{i} := \mathfrak{p}^{\lambda}_{i}(x)$, and $\mathfrak{q}^{\lambda}_{i} := \mathfrak{q}^{\lambda}_{i}(x)$ are real-valued and the $\mathfrak{b}^{\lambda}_{i} = ( \mathfrak{b}^{\lambda \rho}_{i}(x))_{\rho=1}^{d_{1}}$ are $\mathbf{R}^{d_{1}}$-valued $\mathcal{F}_{i} \otimes \mathcal{B}(\mathbf{R}^{d})$-measurable functions for every $\omega \in \Omega$, $x \in \mathbf{R}^{d}$, and $i \in \{0, \dots, n\}$.

In order for $v^{h}$ to approximate the solution of \eqref{eqn: time scheme} in space we require the following \emph{consistency condition}, ensuring the difference operators converge to the differential operators.

\begin{asm}\label{asm: consistency condition} 
For all $\alpha,\beta \in \{1, \dots, d\}$ and $\rho \in \{1, \dots, d_{1}\}$,
\begin{gather*}
\sum_{\lambda \in \Lambda_{0}} \mathfrak{b}^{\lambda \rho}_{i} \lambda^{\alpha} = b^{\alpha \rho}_{i}, \quad \mathfrak{b}^{0\rho}_{i} = b^{0\rho}_{i}, \quad \sum_{\lambda,\mu \in \Lambda_{0}} \mathfrak{a}^{\lambda \mu}_{i} \lambda^{\alpha} \mu^{\beta} = a^{\alpha \beta}_{i}, \quad \mathfrak{a}^{00}_{i} = a^{00}_{i},
\intertext{and}
\sum_{\lambda \in \Lambda_{0}} \mathfrak{a}^{\lambda 0}_{i} \lambda^{\alpha} + \sum_{\mu \in \Lambda_{0}} \mathfrak{a}^{0 \mu}_{i} \mu^{\alpha} + \sum_{\lambda \in \Lambda_{0}} \mathfrak{p}^{\lambda}_{i}\lambda^{\alpha} - \sum_{\mu \in \Lambda_{0}} \mathfrak{q}^{\mu}_{i} \mu^{\alpha} = a^{\alpha 0}_{i} + a^{0 \alpha}_{i}
\end{gather*}
for $i \in \{0,1, \dots, n\}$.
\end{asm}

We also place the following additional assumptions on the coefficients of the difference operators.

\begin{asm}\label{asm: parabolicty condition for difference operators}
For all $\omega \in \Omega$, $x \in \mathbf{R}^{d}$, and $i \in \{0, \dots, n\}$:
\begin{enumerate}
\item  the functions $\mathfrak{p}^{\lambda}\geq 0$ and $\mathfrak{q}^{\lambda} \geq 0$ for all $\lambda \in \Lambda_{0}$;

\item for integer $d_{2} \geq 1$ and $\lambda \in \Lambda_{0}$ there exist $\mathcal{F}_{i} \otimes \mathcal{B}(\mathbf{R}^{d})$-measurable real valued functions $\sigma^{\lambda 1}, \dots, \sigma^{\lambda d_{2}}$ such that 
\begin{equation}\label{eqn: parabolicity condition for difference operators}
\tilde{\mathfrak{a}}^{\lambda \mu}_{i} := 2 \mathfrak{a}^{\lambda \mu}_{i} - \mathfrak{b}^{\lambda \rho}_{i} \mathfrak{b}^{\mu \rho}_{i} = \sum_{r = 1}^{d_{2}} \sigma^{\lambda r}_{i} \sigma^{\mu r}_{i}
\end{equation}
for all $\lambda, \mu \in \Lambda_{0}$. 
\end{enumerate}
\end{asm}

\begin{asm}\label{asm: coefficients of the difference operators}
Let $l \geq 1$ be an integer. For all $\omega \in \Omega$, $i \in \{0, \dots, n\}$, $\lambda \in \Lambda_{0}$, and $k \in \{1, \dots, d_{2}\}$, the functions $\mathfrak{b}^{\lambda}_{i}$ and $\mathfrak{b}^{0}_{i}$ are $l+2$ times continuously differentiable in $x$; the functions $\sigma^{\lambda k}_{i}$ are $l+1$ times continuously differentiable in $x$; and the functions $\mathfrak{a}^{0\lambda}_{i}$, $\mathfrak{a}^{\lambda 0}_{i}$, $\mathfrak{a}^{00}_{i}$, $\mathfrak{p}^{\lambda}_{i}$, and $\mathfrak{q}^{\lambda}_{i}$ are $l$ times continuously differentiable in $x$. Further, there exist constants $\hat{K}_{j}$, for $j \in  \{0, \dots, l+2\}$, such that
\begin{align*}
\left|D^{j} \mathfrak{b}^{\lambda}_{i}\right| + \left|D^{j} \mathfrak{b}^{0}_{i}\right| &\leq \hat{K}_{j} &&\text{for $j \leq l + 2$,}\\
\left|D^{j} \sigma^{\lambda k}_{i}\right| &\leq \hat{K}_{j} &&\text{for $j \leq l + 1$, and} \\
\left|D^{j}\mathfrak{a}^{\lambda 0}_{i}\right| + \left|D^{j} \mathfrak{a}^{0\lambda}_{i} \right| + \left|D^{j} \mathfrak{a}^{00}_{i} \right| + \left|D^{j} \mathfrak{p}^{\lambda}_{i} \right| + 
\left|D^{j} \mathfrak{q}^{\lambda}_{i} \right| & \leq \hat{K}_{j} &&\text{for $j \leq l$}
\end{align*}
for all $\omega \in \Omega$, $x \in \mathbf{R}^{d}$, $i \in \{0, \dots, n\}$, $\lambda \in \Lambda_{0}$, and $k \in \{1, \dots, d_{2}\}$.
\end{asm}

\begin{rmk}
It is clear that \eqref{eqn: parabolicity condition for difference operators} implies that $$\sum_{\lambda, \mu \in \Lambda_{0}} \tilde{\mathfrak{a}}^{\lambda \mu}_{i} z_{\lambda} z_{\mu} \geq 0$$
for $\omega \in \Omega$, $x \in \mathbf{R}^{d}$, $i \in \{0, \dots, d\}$, and $z_{\lambda} \in \mathbf{R}$ for $\lambda \in \Lambda_{0}$. This observation, together with Assumption \ref{asm: consistency condition}, implies Assumption \ref{asm: stochastic parabolicity}. 
\end{rmk}

Solutions to \eqref{eqn: space-time scheme} are understood as sequences of random fields taking values in $\ell^{2}(\mathbf{G}_{h})$, the space of square summable functions on the grid points $\mathbf{G}_{h}$, satisfying \eqref{eqn: space-time scheme} with an $\ell^{2}(\mathbf{G}_{h})$-valued initial condition. The following is a well known result which we include for the sake of completeness. Note that by Assumption \ref{asm: on initial conditions and free terms} the $v_{0}$, $f$ and $g^{\rho}$, for $\rho \in \{1, \dots, d_{1}\}$, are $\ell^{2}(\mathbf{G}_{h})$-valued processes when restricted to the grid $\mathbf{G}_{h}$. 

\begin{thm}\label{thm: existence of l-2 valued solution to space-time scheme}
If Assumptions \ref{asm: on initial conditions and free terms} and \ref{asm: coefficients of the difference operators} hold, then \eqref{eqn: space-time scheme} admits a unique $\ell^{2}(\mathbf{G}_{h})$-valued solution for sufficiently small $\tau$. 
\end{thm}

\begin{proof}
The proof of this solvability result relies on the invertibility of $(I - \tau L^{h}_{i})$, for each $i \in \{0, \dots, n\}$, in $\ell^{2}(\mathbf{G}_{h})$ for sufficiently small $\tau$. Rewriting the scheme as a recursion and using this fact one can construct a unique solution to the scheme iteratively. Full details can be found, for example, in \cite{Hall:2012}.
\end{proof}
%

We observe, however, that \eqref{eqn: space-time scheme} is well defined not only at the points of the grid but for the whole space. Therefore, we consider \eqref{eqn: space-time scheme} on $\mathbf{R}^{d}$ and seek solutions that are sequences of $L^{2}$-valued functions. Hence we will use the normal machinery from analysis to obtain estimates in appropriate Sobolev spaces for solutions to the space-time scheme. Then we will obtain continuous versions of these solutions, by Sobolev's embedding, and show that these solutions agree with the ``natural'' solutions at the grid points. 

To aid in achieving this goal one has the following lemma regarding the embedding $W^{l}_{2} \subset \ell^{2}(\mathbf{G}_{h})$, the proof of which can be found, for example, in \cite{GyongyKrylov:2010}. Recall by Sobolev's embedding of $W^{l}_{2}$ into $\mathcal{C}_{b}$, for $l > d/2$ there exists a linear operator $\mathcal{I} : W^{l}_{2} \to \mathcal{C}_{b}$ such that $\mathcal{I}\phi(x) = \phi (x)$ for almost every $x\in \mathbf{R}^{d}$ and 
\begin{equation*}
\sup_{x \in \mathbf{R}^{d}} \left| \mathcal{I} \phi (x) \right| \leq N \left\| \phi \right\|_{l}
\end{equation*}
for all $\phi \in W^{l}_{2}$ where $N = N(d)$.

\begin{lem}\label{lem: embedding}
For all $\phi \in W^{l}_{2}$ if $l > d/2$ and $h \in (0,1)$, then 
\begin{equation*}
\sum_{x \in \mathbf{G}_{h}} \left| \mathcal{I} \phi (x) \right|^{2} h^{d} \leq N \left\| \phi \right\|_{l}^{2}
\end{equation*}
for a constant $N = N(d)$.
\end{lem}

With these preliminary considerations in mind, we turn to the main pursuit of this paper. To accelerate the rate of convergence of the spatial approximation to an arbitrarily high order via Richardson's method we must first prove the existence of an asymptotic expansion in powers of the discretization parameter $h$ for the error between the space-time approximation and the temporal discretization. Thus we prove that for an integer $k \geq 0$ there exists random fields $v^{(0)}_{i}(x)$, \dots, $v^{(k)}_{i}(x)$ that are independent of $h$ and satisfy certain properties for all $i \in \{0, \dots, n\}$ and $x \in \mathbf{G}_{h}$. Namely, that $v^{(0)}$ is the solution to \eqref{eqn: time scheme} with initial condition $v_{0}$ and for nonzero $h$, 
\begin{equation}\label{eqn: expansion}
v^{h}_{i} (x) = \sum_{j = 0}^{k} \frac{h^{j}}{j!} v^{(j)}_{i}(x) + R^{\tau,h}_{i}(x)
\end{equation}
holds almost surely for all $x \in \mathbf{G}_{h}$ and all $i \in \{0, \dots, n\}$ where $v^{h}$ is the solution to \eqref{eqn: space-time scheme} with initial condition $v_{0}$ and $R^{h}$ is an $\ell^{2}(\mathbf{G}_{h})$-valued adapted process such that 
\begin{equation}\label{eqn: estimate for error term}
\mexp \max_{i \leq n} \sup_{x \in \mathbf{G}_{h}} \left| R^{\tau,h}_{i}(x) \right|^{2} \leq N h^{2(k+1)} \mathcal{K}^{2}_{m}
\end{equation}
for a constant $N$ independent of $h$ and $\tau$. 

We include the following additional assumption on the coefficients of the difference operators because at certain points in the proofs to come we will require less regularity than is guaranteed by Assumption \ref{asm: coefficients of the difference operators}.

\begin{asm}\label{asm: coefficients of the difference operators, less regularity}
Let $\mathfrak{m} \geq 0$ be a fixed integer. For $\lambda, \mu \in \Lambda$, the spatial derivatives of $\mathfrak{a}^{\lambda \mu}_{i}$ and $\mathfrak{b}^{\lambda}_{i}$ exist up to order $(\mathfrak{m}-4)\vee 0$ and for $\lambda \in \Lambda_{0}$ the spatial derivatives of $\mathfrak{p}^{\lambda}_{i}$ and $\mathfrak{q}^{\lambda}_{i}$ exist up to order $(\mathfrak{m}-2)\vee 0$ and the coefficients together with their derivatives are bounded by constants $C_{\mathfrak{m}}$ for all $\omega \in \Omega$, $x \in \mathbf{R}^{d}$, and $i \in \{0, \dots, n\}$. 
\end{asm}

\begin{thm}\label{thm: expansion}
If Assumption \ref{asm: coefficients of the difference operators} holds with integer $l \geq d/2$ and Assumptions \ref{asm: stochastic parabolicity}, \ref{asm: coefficients of the differential operators}, \ref{asm: on initial conditions and free terms}, \ref{asm: consistency condition}, \ref{asm: parabolicty condition for difference operators}, and  \ref{asm: coefficients of the difference operators, less regularity} hold with
\begin{equation}\label{eqn: condition on m, expansion}
m = \mathfrak{m} \geq 3k + 4 + l
\end{equation}
for integer $k \geq 0$, then expansion \eqref{eqn: expansion} and estimate \eqref{eqn: estimate for error term} hold for $h > 0$ with a constant $N = N(d,d_{1},d_{2},m,l,T,K_{0}, \dots, K_{m+2}, \hat{K}_{0}, \dots, \hat{K}_{l+2}, C_{m},\Lambda)$. If, in addition, $\mathfrak{p}^{\lambda} = \mathfrak{q}^{\lambda} = 0$ for $\lambda \in \Lambda_{0}$, then \eqref{eqn: expansion} and \eqref{eqn: estimate for error term} hold for all nonzero $h$. In this case, the $v^{(j)}$ vanish for odd $j \leq k$ and, hence, if $k$ is odd, then \eqref{eqn: condition on m, expansion} can be replaced with $\mathfrak{m} = m \geq 3k + 1 + l$.
\end{thm}

This theorem follows from the next result, which will also allow us to provide higher order estimates for derivatives of the solutions. Taking differences of \eqref{eqn: expansion} yields 
\begin{equation*}
\delta_{h,\lambda} v^{h}_{i}(x) + \sum_{j=0}^{k} \frac{h^{j}}{j!} \delta_{h,\lambda} v^{(j)}_{i} (x) + \delta_{h,\lambda} R^{\tau, h}_{i}(x)
\end{equation*}
for any $\lambda := (\lambda_{1},\dots,\lambda_{p}) \in \Lambda^{p}$, for integer $p \geq 0$, where $\Lambda^{0} := \{0\}$ and $\delta_{h,\lambda} := \delta_{h,\lambda_{1}}\times \cdots \times \delta_{h,\lambda_{p}}$. Although the estimate for $\delta_{h,\lambda} R^{\tau, h}_{i}(x)$ is not obvious, we have the following generalization of Theorem \ref{thm: expansion}.

\begin{thm}\label{thm: generalized expansion}
Let the assumptions of Theorem \ref{thm: expansion} hold with 
\begin{equation}\label{eqn: condition on m, generalized expansion}
\mathfrak{m} = m \geq p + 3k + 4 + l
\end{equation} 
for integers $l > d/2$, $p \geq 0$, and $k \geq 0$ with $\lambda \in \Lambda^{p}$. Then for $h > 0$ expansion \eqref{eqn: expansion} and 
\begin{equation*}
\mexp \max_{i \leq n} \sup_{x \in \mathbf{G}_{h}} \left| \delta_{h,\lambda} R^{\tau,h}_{i}(x) \right|^{2} \leq N h^{2(k+1)} \mathcal{K}^{2}_{m},
\end{equation*}
hold for a constant $N = N(p,d,d_{1},d_{2},m,l,T,K_{0}, \dots, K_{m+2}, \hat{K}_{0}, \dots, \hat{K}_{l+2}, C_{m},\Lambda)$. If, in addition, $\mathfrak{p}^{\lambda} = \mathfrak{q}^{\lambda} = 0$ for $\lambda \in \Lambda_{0}$, then the terms $v^{(j)}$ vanish for odd $j \leq k$ and, therefore, if $k$ is odd, then \eqref{eqn: condition on m, generalized expansion} can be replaced with $m \geq p + 3k + 1 + l$. 
\end{thm}

This theorem and Theorem \ref{thm: expansion} follow from a more general result that is proven in Section \ref{sec: Proof of Main Results} after some preliminaries are presented in Section \ref{>>sec: Auxilliary Results}. Presently we formulate our acceleration result, which says the rate of convergence of the spatial approximation can be accelerated to an arbitrarily high order by taking suitable weighted averages of the approximation at different mesh sizes.

Fix an integer $k \geq 0$ and let 
\begin{gather}\label{eqn: v-bar and v-tilde} 
\bar{v}^{h} := \sum_{j=0}^{k} \bar{\beta}_{j} v^{2^{-j} h} \quad \text{and} \quad 
\tilde{v}^{h} := \sum_{j=0}^{\tilde{k}} \tilde{\beta}_{j} v^{2^{-j}h}
\end{gather}
where $v^{2^{-j}h}$ solves, with $2^{-j}h$ in place of $h$, the space-time scheme \eqref{eqn: space-time scheme} with initial condition $v_{0}$. Here $\bar{\beta}$ is given by $(\bar{\beta}_{0}, \bar{\beta}_{1}, \dots, \bar{\beta}_{k}) := (1, 0, \dots, 0) \bar{V}^{-1} $
where $\bar{V}^{-1}$ is the inverse of the Vandermonde matrix with entries $\bar{V}^{ij} = 2^{-(i-1)(j-1)}$ for $i,j \in \{1, \dots, k+1\}$. Similarly, $\tilde{\beta}$ is given by $(\tilde{\beta}_{0}, \tilde{\beta}_{1}, \dots, \tilde{\beta}_{k}) := (1, 0, \dots, 0) \tilde{V}^{-1}$ where $\tilde{V}^{-1}$ s the inverse of the Vandermonde matrix with entries $\tilde{V}^{ij} = 4^{-(i-1)(j-1)}$ for $i,j \in \{1, \dots, \tilde{k} +1\}$ where $\tilde{k} := \left\lfloor \frac{k}{2} \right\rfloor$. Here $\left\lfloor c \right\rfloor$ denotes the integer part of $c$. Recall that $v^{(0)}$ is the solution to \eqref{eqn: time scheme} with initial condition $v_{0}$.

\begin{thm}\label{thm: acceleration}
Let the assumptions of Theorem \ref{thm: expansion} hold with
\begin{equation}\label{eqn: condition on m, acceleration}
\mathfrak{m} = m \geq 3k + 4 + l 
\end{equation}
for integers $l > d/2$ and $k \geq 0$. Then 
\begin{equation}\label{eqn: acceleration, h>0}
\mexp \max_{i \leq n} \sup_{x \in \mathbf{G}_{h}} \left| \bar{v}^{h}_{i}(x) - v^{(0)}_{i}(x) \right|^{2} \leq N h^{2(k+1)} \mathcal{K}^{2}_{m}
\end{equation}
holds for $h > 0$ with $N = N(d,d_{1},d_{2},m,l,T,K_{0}, \dots, K_{m+2}, \hat{K}_{0}, \dots, \hat{K}_{l+2}, C_{m},\Lambda)$. If, in addition, $\mathfrak{p}^{\lambda} = \mathfrak{q}^{\lambda} = 0$ for $\lambda \in \Lambda_{0}$, then 
\begin{equation}\label{eqn: acceleration, h neq 0}
\mexp \max_{i\leq n} \sup_{x \in \mathbf{G}_{h}} \left| \tilde{v}^{h}_{i}(x) - v^{(0)}_{i}(x) \right|^{2} \leq N |h|^{2(k+1)} \mathcal{K}^{2}_{m}
\end{equation}
holds for nonzero $h$. Moreover, if $k$ is odd, then we only require $\mathfrak{m} = m \geq 3k + 1 + l$ in place of \eqref{eqn: condition on m, acceleration}.
\end{thm}

\begin{proof}
By Theorem \ref{thm: expansion}, 
we have the expansion 
$$v^{2^{-j}h} = v^{(0)} + \sum_{i=1}^{\tilde{k}} \frac{h^{2i}}{2i! 4^{ij}} v^{(2i)} + r^{2^{-j}h} h^{\tilde{k}+1}$$
for each $j \in \{0,1,\dots, k\}$ where $r^{2^{-j}h} := h^{-(\tilde{k}+1)} R^{2^{-j}h}$. Then for $\tilde{r}^{h} := \sum_{j=0}^{\tilde{k}} r^{2^{-j} h}$,
\begin{align*}
\tilde{v}^{h}
	&= \left( \sum_{j=0}^{\tilde{k}} \tilde{\beta}_{j} \right) v^{(0)} + \sum_{j=0}^{\tilde{k}} \sum_{i=1}^{\tilde{k}} \tilde{\beta}_{j} \frac{h^{2i}}{2i!4^{ij}} v^{(2i)} + \tilde{r}^{h} h^{k+1}\\
	&= v^{(0)} + \sum_{i=1}^{\tilde{k}} \frac{h^{2i}}{2i!} v^{(2i)} \sum_{j=0}^{\tilde{k}} \frac{\tilde{\beta}_{j}}{4^{ij}} + \tilde{r}^{h} h^{k+1}\\
	&= v^{(0)} + \tilde{r}^{h}h^{k+1},
\end{align*}
since $\sum_{j=0}^{\tilde{k}} \tilde{\beta}_{j} = 1$ and $\sum_{j=0}^{\tilde{k}} \tilde{\beta}_{j} 4^{-ij} = 0$ for each $i \in \{ 1,2, \dots, k\}$ by the definition of $\tilde{\beta}$. Now using the bound on $R^{\tau, h}$ from Theorem \ref{thm: expansion} together with this last calculation yields \eqref{eqn: acceleration, h neq 0}. The result for \eqref{eqn: acceleration, h>0} is obtained in an almost identical way and therefore we omit the proof.
\end{proof}

\begin{rmk}
Note that without the acceleration, that is, when $k = 0$ and $k =1$ in \eqref{eqn: acceleration, h>0} and \eqref{eqn: acceleration, h neq 0}, respectively, we have that 
\begin{equation*}
\mexp \max_{i \leq n} \sup_{x \in \mathbf{G}_{h}} \left| v^{h}_{i}(x) - v_{i}(x) \right|^{2} \leq Nh^{2} \mathcal{K}^{2}_{m}
\end{equation*}	
and if $\mathfrak{p}^{\lambda} = \mathfrak{q}^{\lambda} = 0$ for $\lambda \in \Lambda_{0}$, then we have 
\begin{equation*}
\mexp \max_{i \leq n} \sup_{x \in \mathbf{G}_{h}} \left| v^{h}_{i}(x) - v_{i}(x) \right|^{2} \leq N h^{4} \mathcal{K}^{2}_{m}
\end{equation*}
in the theorem above. Moreover, these estimates are sharp; see Remark 2.21 in \cite{DongKrylov:2005} on finite difference approximations for deterministic parabolic partial differential equations. 
\end{rmk}

One can also construct rapidly converging approximations for the derivatives of $v^{(0)}$ by taking suitable weighted averaged of finite differences of $\tilde{v}^{h}$.

\begin{thm}\label{thm: acceleration for derivatives}
Let $p \geq 0$ be an integer and let $\mathfrak{p}^{\lambda} = \mathfrak{q}^{\lambda} = 0$ for $\lambda \in \Lambda_{0}$. If the assumptions of Theorem \ref{thm: expansion} hold with $$\mathfrak{m} = m \geq p + 3k + 4 + l,$$ for integers $l > d/2$, $k \geq 0$, and $p \geq 0$, then for $\lambda \in \Lambda^{p}$ equation \eqref{eqn: acceleration, h neq 0} holds with $\delta_{h,\lambda}\tilde{v}^{h}$ and $\delta_{h,\lambda} v^{(0)}$ in place of $\tilde{v}^{h}$ and $v^{(0)}$ respectively. 
\end{thm}

\begin{proof}
This assertion follows from Theorem \ref{thm: generalized expansion} in exactly the same way that Theorem \ref{thm: acceleration} follows from \ref{thm: expansion}.
\end{proof}

We end this section with two examples of ways to choose appropriate $\mathfrak{a}$, $\mathfrak{b}$, $\mathfrak{p}$, $\mathfrak{q}$ and $\Lambda$. 

\begin{eg}
Let $\Lambda = \{ e_{0}, e_{1}, \dots, e_{d} \}$ where $e_{0} = 0$ and $e_{i}$ is the $i$th basis vector, that is, $\Lambda$ is the basis vectors in $\mathbf{R}^{d}$ together with the origin. Then for $i \in \{0,1, \dots, n\}$, set
\begin{gather*}
\mathfrak{a}^{e_{\alpha}e_{\beta}}_{i} := a^{\alpha\beta}_{i} \quad \text{and} \quad \mathfrak{b}^{e_{\alpha}\rho}_{i} := b^{\alpha\rho}_{i}
\end{gather*}
for each $\alpha, \beta \in \{0, 1, \dots, d\}$ and 
\begin{gather*}
\mathfrak{p}^{e_{\alpha}}_{i} = \mathfrak{q}^{e_{\alpha}}_{i} := 0
\end{gather*}
for each $\alpha \in \{1, \dots, d\}$. Then each spatial derivative $D_{\alpha}$ in \eqref{eqn: time scheme} is approximated by the symmetric difference $\delta^{h}_{e_{\alpha}}$.
\end{eg}

\begin{eg}
Let $\Lambda$ again be the basis vectors in $\mathbf{R}^{d}$ together with the origin. For $i \in \{0, \dots, n\}$, set 
\begin{gather*}
\mathfrak{a}^{00}_{i} := a^{00} \quad \text{and} \quad \mathfrak{a}^{e_{\alpha}e_{\beta}}_{i} := a^{\alpha\beta}_{i}
\end{gather*}
for each $\alpha,\beta \in \{1, \dots, d\}$ and
\begin{gather*}
\mathfrak{b}^{e_{\alpha}\rho}_{i} := b^{\alpha\rho}_{i}
\end{gather*}
for $\alpha \in \{0, \dots, d\}$. We also take $\mathcal{F}_{i} \otimes \mathcal{B}(\mathbf{R}^{d})$-measurable functions $\mathfrak{p}^{e_{\alpha}}$ and $\mathfrak{q}^{e_{\alpha}}$ for $\alpha \in \{1, \dots, d\}$ such that
\begin{gather*}
\mathfrak{p}^{e_{\alpha}}_{i} - \mathfrak{q}^{e_{\alpha}}_{i} := a^{0\alpha}_{i} + a^{\alpha 0}_{i}
\end{gather*}
for $\alpha \in \{1, \dots, d\}$. 
\end{eg}

In the next section we make observations that will be used in the proofs of Theorems \ref{thm: generalized expansion} and \ref{thm: expansion} which are given in Section \ref{sec: Proof of Main Results}.

\section{Auxiliary Results}\label{>>sec: Auxilliary Results}

We begin by delivering a proof for Theorem \ref{thm: solvability of the time scheme with estimate}. For integer $m \geq 0$, recall Lemma 2.1 from \cite{KrylovRozovskii:1982} taking the parameter $p$ in the Lemma to be $p = 2$. This Lemma holds for all $t \in [0,T]$ so in particular we have it for each $i \tau$ for $i \in \{0, \dots, n\}$.  

\begin{lem}\label{lem: estimate from Krylov and Rozovskii (1982)}
Let $\phi \in W^{m+2}_{2}$. If Assumptions \ref{asm: stochastic parabolicity} and \ref{asm: coefficients of the differential operators} hold for all multiindices $\gamma$ such that $|\gamma| \leq m$, then 
\begin{equation*}
\begin{split}
\mathcal{Q}^{\gamma}_{i} (\phi) := \int_{\mathbf{R}^{d}} 2 \left( D^{\gamma} \phi \right) D^{\gamma} \mathcal{L}_{i} \phi + \sum_{\rho=1}^{d_{1}} \left| D^{\gamma}\mathcal{M}^{\rho}_{i} \phi \right|^{2} \mathrm{d}x \leq N \left\| \phi \right\|_{m}^{2}
\end{split}
\end{equation*}
for a constant $N = N(d,d_{1},m,K_{0},\dots,K_{m})$.
\end{lem}

We use Lemma \ref{lem: estimate from Krylov and Rozovskii (1982)} to obtain estimate \eqref{eqn: estimate for the time scheme}. The existence of a solution to \eqref{eqn: time scheme} will follow from the vanishing viscosity method. 

\begin{proof}[Proof of Theorem \ref{thm: solvability of the time scheme with estimate}.]
We first assume that a sufficiently smooth solution to \eqref{eqn: time scheme} exists and obtain estimate \eqref{eqn: estimate for the time scheme} for a constant $N$ independent of $\tau$. We begin by obtaining an expression for the square of the norm for the solution to the time scheme. Then we estimate the supremum of the expectation of the square of the norm and in particular we show that this quantity is finite. With these observations in place we are then able to estimate the expectation of the supremum of the square of the norm.

For a multiindex $|\gamma| \leq m$, considering the equality $a^{2} - b^{2} = 2a(a-b) - |a-b|^{2}$ we note that \eqref{eqn: space-time scheme} implies
\begin{align*}
\left\| D^{\gamma} v_{i}\right\|_{0}^{2} - \left\| D^{\gamma} v_{i-1} \right\|_{0}^{2} &= 2 \left(D^{\gamma} v_{i}, D^{\gamma} \left(v_{i} - v_{i-1}\right)\right) - \left\|D^{\gamma}\left(v_{i} - v_{i-1}\right)\right\|_{0}^{2}\\
	& = 2\left(D^{\gamma} v_{i}, D^{\gamma}\left(\mathcal{L}_{i}v_{i} + f_{i}\right)\right)\tau - \left\|D^{\gamma}\left(v_{i} - v_{i-1}\right) \right\|_{0}^{2} \\ & \quad + 2 \sum_{\rho=1}^{d_{1}}\left(D^{\gamma}v_{i-1}, D^{\gamma}\left(\mathcal{M}^{\rho}_{i-1}v_{i-1} + g^{\rho}_{i-1}\right)\right)\xi^{\rho}_{i}\\
	& \quad + 2\sum_{\rho=1}^{d_{1}}\left(D^{\gamma}\left(v_{i} - v_{i-1}\right),D^{\gamma}\left(\mathcal{M}^{\rho}_{i-1}v_{i-1} + g^{\rho}_{i-1}\right)\right)\xi^{\rho}_{i}\\
	&= 2\left(D^{\gamma} v_{i}, D^{\gamma}\left(\mathcal{L}_{i}v_{i} + f_{i}\right)\right)\tau - \left\|D^{\gamma}\left(\mathcal{L}_{i} v_{i} + f_{i}\right)\right\|_{0}^{2} \tau^{2} \\ &\quad + 2\sum_{\rho=1}^{d_{1}}\left(D^{\gamma}v_{i-1}, D^{\gamma}\left(\mathcal{M}^{\rho}_{i-1}v_{i-1} + g^{\rho}_{i-1}\right)\right)\xi^{\rho}_{i}\\
	& \quad + \left\| \sum_{\rho=1}^{d_{1}} D^{\gamma}\left(\mathcal{M}^{\rho}_{i-1}v_{i-1} + g^{\rho}_{i-1}\right) \xi^{\rho}_{i}\right\|_{0}^{2}.
\end{align*} 
Summing up over $i$ from $1$ to $j$ and over $|\gamma| \leq m$, we have 
\begin{equation}
\label{eqn: continuous intermediate equation with H, I, J}
\left\| v_{j} \right\|_{m}^{2} \leq \left\| v_{0}\right\|_{m}^{2} + H_{j} + I_{j} + J_{j},
\end{equation}
where 
\begin{equation*}
H_{j} := 2 \sum_{i=1}^{j} \left(D^{m}v_{i}, D^{m}\left(\mathcal{L}_{i} v_{i} + f_{i} \right)\right)\tau,
\end{equation*}
\begin{equation*}
I_{j} := 2 \sum_{i=1}^{j} \sum_{\rho=1}^{d_{1}} \left(D^{m}v_{i-1}, D^{m}\left(\mathcal{M}^{\rho}_{i-1} v_{i-1} + g^{\rho}_{i-1} \right)\right)\xi^{\rho}_{i},
\end{equation*}
and 
\begin{equation*}
J_{j} := \sum_{i=1}^{j} \left\| \sum_{\rho =1}^{d_{1}} D^{m} \left(\mathcal{M}^{\rho}_{i-1} v_{i-1} + g^{\rho}_{i-1}\right) \xi^{\rho}_{i} \right\|_{0}^{2}.
\end{equation*}

By an application of It{\^o}'s formula, for each $\pi,\rho \in \{1, \dots, d_{1}\}$ one has that for all $i \in \{0, \dots, n-1\}$
\begin{equation*}
\xi^{\pi}_{i+1}\xi^{\rho}_{i+1} = (w^{\pi}_{i+1} - w^{\pi}_{i}) (w^{\rho}_{i+1} - w^{\rho}_{i}) = Y^{\pi \rho}_{i+1} - Y^{\pi \rho}_{i} + \tau \chi_{\pi \rho}
\end{equation*}
for
$$Y^{\pi \rho}(t) := \int_{0}^{t} \left(w^{\pi}(s) - w^{\pi}_{\kappa(s)} \right)\mathrm{d}w^{\rho}(s) + \int_{0}^{t} \left(w^{\rho}(s) - w^{\rho}_{\kappa(s)}\right)\mathrm{d}w^{\pi}(s)$$ 
where $\kappa(s)$ is the piecewise defined function taking value $\kappa(s) = i$ for $s \in [i\tau,(i+1)\tau)$ and where $\chi_{\pi\rho} = 1$ when $\pi = \rho$ and $0$ otherwise. Thus we can write $J_{j} = J^{(1)}_{j} + J^{(2)}_{j}$ where
\begin{equation*}
J^{(1)}_{j} := \sum_{i=1}^{j} \left\| \sum_{\rho=1}^{d_{1}} D^{m}\left( \mathcal{M}^{\rho}_{i-1} v_{i-1} + g^{\rho}_{i-1} \right)\right\|_{0}^{2} \tau
\end{equation*}
and
\begin{equation*}
J^{(2)}_{j} := \int_{0}^{j\tau} \sum_{\pi,\rho = 1}^{d_{1}} \left(D^{m}\left(\mathcal{M}^{\pi}_{\kappa(s)} v_{\kappa(s)} + g^{\pi}_{\kappa(s)}\right), D^{m}\left(\mathcal{M}^{\rho}_{\kappa(s)} v_{\kappa(s)} + g^{\rho}_{\kappa(s)}\right)\right) \mathrm{d} Y^{\pi \rho}(s).
\end{equation*}

Now observe that, for each $i \in \{1, \dots, n\}$, by Lemma \ref{lem: estimate from Krylov and Rozovskii (1982)} we have
\begin{align}
H_{j} + J^{(1)}_{j} 
	&\leq N \tau \sum_{\alpha = 0}^{d} \left\| D_{\alpha} v_{0}\right\|_{m}^{2} + N \tau \sum_{i=1}^{j} \sum_{|\gamma|\leq m} \left( \mathcal{Q}^{\gamma}_{i} (v_{i}) + \left\| D^{\gamma} f_{i}\right\|_{0}^{2} + \left\|D^{\gamma} g_{i-1}\right\|_{0}^{2}\right) \notag\\
	&\leq N \tau \left\| v_{0}\right\|_{m+1}^{2} + N \tau \sum_{i=1}^{j} \left(\left\|v_{i}\right\|_{m}^{2} +\left\| f_{i} \right\|_{m}^{2} + \left\| g_{i-1}\right\|_{m}^{2}\right) \label{eqn: for the remark},
\end{align}
where $N = N (d,d_{1},m,K_{0}, \dots, K_{m+1})$. Note that we only require $b^{\alpha}$ to have bounded derivatives of at most order $m+1$; the other coefficients only need to have bounded derivatives of at most order $m$ at this stage. Here the initial condition $v_{0}$ enters, estimated in the $W^{m+1}_{2}$-norm, due to the displacement caused by the discretization in time when we consider the quadratic form $\mathcal{Q}^{\gamma}$ from Lemma \ref{lem: estimate from Krylov and Rozovskii (1982)}. Thus inequality \eqref{eqn: continuous intermediate equation with H, I, J} becomes
\begin{equation}
\label{eqn: continuous intermediate equation after applying Q bound }
\left\| v_{j} \right\|_{m}^{2} \leq N \tau \left\|v_{0}\right\|_{m+1}^{2} + N \tau \sum_{i=1}^{j}\left( \left\|v_{i}\right\|_{m}^{2} + \left\| f_{i} \right\|_{m}^{2} + \left\| g_{i-1}\right\|_{m}^{2}\right) + I_{j} + J^{(2)}_{j}.
\end{equation}
Since $v_{i}$, $\mathcal{M}^{\rho}_{i} v_{i}$, and $g^{\rho}_{i}$ are all $\mathcal{F}_{i}$-measurable and $\xi^{\rho}_{i+1}$ is independent of $\mathcal{F}_{i}$ for $i \in \{0, \dots, n\}$, we have that
$$ \mexp I_{j} = 2 \sum_{i=1}^{j} \sum_{\rho=1}^{d_{1}} \mexp \left\{ \left( D^{\gamma} v_{i-1}, D^{\gamma}\left(\mathcal{M}^{\rho}_{i-1} v_{i-1} + g^{\rho}\right) \right) \mexp \left( \xi^{\rho}_{i} \mid \mathcal{F}_{i-1} \right) \right\} = 0.$$
Similarly, we see that $\mexp J^{(2)}_{j} = 0$ since the expectation of the stochastic integral is zero. Therefore, taking the expectation of \eqref{eqn: continuous intermediate equation after applying Q bound } and the sum of $f$ and $g$ over $i \in \{0, \dots, n\}$, we have that
\begin{equation}
\label{eqn: continuous intermediate equation after taking the expectation and maximum}
\mexp \left\| v_{j} \right\|_{m}^{2}  \leq N \left(\tau \mexp \left\|v_{0}\right\|_{m+1}^{2} +\left\llbracket f \right\rrbracket_{m}^{2} + \left\llbracket g \right\rrbracket_{m}^{2}\right) + N \tau \mexp \sum_{i=1}^{j} \left\| v_{i} \right\|_{m}^{2}
\end{equation}
for each $j \in \{1, \dots, n\}$. Applying a discrete Gronwall lemma to \eqref{eqn: continuous intermediate equation after taking the expectation and maximum} we have
$$ \mexp \left\| v_{j}\right\|_{m}^{2} \leq N \left( \tau \mexp \left\| v_{0}\right\|_{m+1}^{2} + \left\llbracket f \right\rrbracket_{m}^{2} + \left\llbracket g \right\rrbracket_{m}^{2}\right) \left(1 - N \tau\right)^{-j}$$
and, since 
$$\left(1 - N\tau\right)^{-j} = \left(1- N \frac{T}{n}\right)^{-j} \leq \left(1 - N \frac{T}{n}\right)^{-n} \leq C e^{NT},$$ 
we have the following estimate for the supremum of the expectation of the square of the norm
\begin{equation}
\label{eqn: continuous intermediate bound on max E v after applying discrete Gronwall lemma}
\max_{i \leq n} \mexp \left\| v_{i}\right\|_{m}^{2} \leq N \left( \tau \mexp \left\| v_{0} \right\|_{m+1}^{2} +  \left\llbracket f \right\rrbracket_{m}^{2} +  \left\llbracket g \right\rrbracket_{m}^{2} \right)
\end{equation}
for a constant $N = N (d,d_{1},m,T, K_{0}, \dots, K_{m+1})$. In particular, we can use \eqref{eqn: continuous intermediate bound on max E v after applying discrete Gronwall lemma} to eliminate the last term on the right-hand side of \eqref{eqn: continuous intermediate equation after taking the expectation and maximum} by bounding it with terms already appearing on the right-hand side \eqref{eqn: continuous intermediate equation after taking the expectation and maximum}. 

Next we approach the estimate for the expectation of the supremum by first observing how to bound the $I$ and $J^{(2)}$ terms appearing in \eqref{eqn: continuous intermediate equation after applying Q bound } using the Burkholder--Davis--Gundy inequality. For $J^{(2)}$ we have
\begin{align*}
\mexp & \max_{i \leq n} \left|J^{(2)}_{i}\right|
	\\ &\leq C \sum_{\pi,\rho=1}^{d_{1}} \mexp \left\{ \int_{0}^{T} \left\|\mathcal{M}^{\rho}_{\kappa(s)} v_{\kappa(s)} + g^{\rho}_{\kappa(s)}\right\|_{m}^{2} \left\| \mathcal{M}^{\pi}_{\kappa(s)}v_{\kappa(s)} + g^{\pi}_{\kappa(s)}\right\|_{m}^{2} \mathrm{d} \left\langle Y^{\pi\rho}\right\rangle(s)\right\}^{1/2}\\
	&\leq C \sum_{\pi,\rho=1}^{d_{1}} \mexp \left\{ \int_{0}^{T} \left\|\mathcal{M}^{\rho}_{\kappa(s)} v_{\kappa(s)} + g^{\rho}_{\kappa(s)}\right\|_{m}^{4} \left|w^{\pi}(s) - w^{\pi}_{\kappa(s)}\right|^{2} \mathrm{d}s \right\}^{1/2}\\
	&\leq C \sum_{\pi,\rho=1}^{d_{1}} \mexp \max_{i \leq n} \sqrt{\tau} \left\|\mathcal{M}^{\rho}_{i} v_{i} + g^{\rho}_{i}\right\|_{m} \\
		&\qquad \times \left\{ \frac{1}{\tau} \int_{0}^{T} \left\|\mathcal{M}^{\rho}_{\kappa(s)} v_{\kappa(s)} + g^{\rho}_{\kappa(s)}\right\|_{m}^{2} \left|w^{\pi}(s) - w^{\pi}_{\kappa(s)} \right|^{2} \mathrm{d}s \right\}^{1/2}
\end{align*}
where $C$ is a constant independent of the parameters and functions under consideration and is allowed to change from one instance to the next. Therefore,
\begin{equation}
\label{eqn: continuous intermediate estimate for E max J-2}
\begin{split}
\mexp \max_{i \leq n} \left| J^{(2)}_{i} \right| \leq& d_{1} C \sum_{\rho=1}^{d_{1}} \tau \mexp \max_{i \leq n} \left\| \mathcal{M}^{\rho}_{i}v_{i} + g^{\rho}_{i}\right\|_{m}^{2} \\ & + \frac{C}{\tau} \sum_{\pi,\rho=1}^{d_{1}} \mexp \int_{0}^{T} \left\| \mathcal{M}^{\rho}_{\kappa(s)} v_{\kappa(s)} + g^{\rho}_{\kappa(s)}\right\|_{m}^{2} \left|w^{\pi}(s) - w^{\pi}_{\kappa(s)}\right|^{2} \, \mathrm{d}s
\end{split}
\end{equation}
by Young's inequality. We observe that the second term on the right-hand side of \eqref{eqn: continuous intermediate estimate for E max J-2} can be estimated by 
\begin{align*}
\frac{1}{\tau}  \sum_{\pi, \rho=1}^{d_{1}} & \mexp  \int_{0}^{T} \left\| \mathcal{M}^{\rho}_{\kappa(s)} v_{\kappa(s)} + g^{\rho}_{\kappa(s)} \right\|_{m}^{2} \left| w^{\pi}(s) - w^{\pi}_{\kappa(s)}\right|^{2} \mathrm{d}s \\
	& \leq \frac{1}{\tau} \sum_{\pi, \rho=1}^{d_{1}} \mexp \left\{ \int_{0}^{T} \left\| \mathcal{M}^{\rho}_{\kappa(s)} v_{\kappa(s)} + g^{\rho}_{\kappa(s)} \right\|_{m}^{2} \mexp \left( \left| w^{\pi}(s) - w^{\pi}_{\kappa(s)}\right|^{2} \;\middle\vert\; \mathcal{F}_{\kappa(s)} \right) \mathrm{d}s \right\} \\
	&\leq N \tau \mexp \sum_{i=0}^{n} \left\| v_{i} \right\|_{m+1}^{2} + N \left\llbracket g \right\rrbracket_{m+1}^{2}
\end{align*}
using the tower property for conditional expectations. Further, the first term on the right-hand side of \eqref{eqn: continuous intermediate estimate for E max J-2} is bounded from above by the sum over all $i \in \{1, \dots, n\}$ and can be estimated by the same quantity. Combining these estimates and using \eqref{eqn: continuous intermediate bound on max E v after applying discrete Gronwall lemma} with $m + 1$ in place of $m$ we see that $\mexp \max | J^{(2)} |$ is estimated by 
\begin{equation}\label{eqn: continuous bound for E max J-2}
\mexp \max_{i \leq n} \left| J^{(2)}_{i} \right| \leq  N \left( \tau \mexp \left\| v_{0} \right\|_{m+2}^{2} + \left\llbracket f \right\rrbracket_{m+1}^{2} + \left\llbracket g\right\rrbracket_{m+1}^{2} \right)
\end{equation}
for a constant $N = N (d,d_{1}, m,T, K_{0}, \dots, K_{m+2})$. Moving on to $I$, we note that
\begin{align*}
I_{j} &= 2 \sum_{i=1}^{j} \sum_{\rho=1}^{d_{1}} \left(D^{m}v_{i-1}, D^{m}\left(\mathcal{M}^{\rho}_{i-1} v_{i-1} + g^{\rho}_{i-1} \right)\right) \left(w^{\rho}_{i} - w^{\rho}_{i-1}\right) \\
	&= 2 \sum_{\rho=1}^{d_{1}} \int_{0}^{j\tau} \left(D^{m} v_{\kappa(s)},D^{m}\left( \mathcal{M}^{\rho}_{\kappa(s)} v_{\kappa(s)} + g^{\rho}_{\kappa(s)}\right)\right) \mathrm{d}w^{\rho}(s).
\end{align*}
Applying the Burkholder--Davis--Gundy inequality once again, we obtain
\begin{align*}
\mexp \max_{i \leq n} \left| I_{i} \right| 
	&\leq C \sum_{\rho=1}^{d_{1}} \mexp \left\{ \int_{0}^{T} \left\| v_{\kappa(s)} \right\|_{m}^{2} \left\| \mathcal{M}^{\rho}_{\kappa(s)} v_{\kappa(s)} + g^{\rho}_{\kappa(s)} \right\|_{m}^{2} \mathrm{d}s \right\}^{1/2}\\
	&\leq C \sum_{\rho=1}^{d_{1}} \mexp \left\{ \max_{i \leq n} \left\|v_{i}\right\|_{m} \left( \int_{0}^{T} \left\| \mathcal{M}^{\rho}_{\kappa(s)}v_{\kappa(s)} + g^{\rho}_{\kappa(s)} \right\|_{m}^{2} \mathrm{d}s \right)^{1/2}\right\}
\end{align*}
and then using Young's inequality followed by \eqref{eqn: continuous intermediate bound on max E v after applying discrete Gronwall lemma} with $m+1$ in place of $m$ we see that $\mexp \max | I |$ is estimated by the same quantity appearing on the right and side of \eqref{eqn: continuous bound for E max J-2}. 

Returning to \eqref{eqn: continuous intermediate equation after applying Q bound }, taking the maximum followed by the expectation, and using the estimates for the expectation of the supremum of $|J^{(2)}|$ and $| I|$, we see that
\begin{equation}\label{eqn: full estimate for the time scheme in the proof}
\mexp \max_{i \leq n} \left\| v_{i} \right\|_{m}^{2} \leq N \left( \tau \mexp \left\| v_{0} \right\|_{m+2}^{2} + \left\llbracket f \right\rrbracket_{m+1}^{2} + \left\llbracket g \right\rrbracket_{m+1}^{2}\right) = N\mathcal{K}_{m}^{2},
\end{equation}
holds with a constant $N = N (d,d_{1},m, T, K_{0}, \dots, K_{m+2})$, thus establishing \eqref{eqn: estimate for the time scheme}. Next, we use the vanishing viscosity method to show that \eqref{eqn: time scheme} admits a solution.

For $\varepsilon > 0$, we let $\mathcal{L}^{\varepsilon}_{i} \phi := \mathcal{L}_{i} \phi + \varepsilon \triangle \phi$ where $\triangle := \sum_{\alpha = 1}^{d} D_{\alpha}D_{\alpha}$ is the Laplacian. Notice the leading coefficient of the operator $\mathcal{L}^{\varepsilon}_{i}$ is given by $\bar{a}^{\alpha\beta}_{i} := a^{\alpha\beta}_{i} + \varepsilon \chi_{\alpha\beta}$, where $\chi_{\alpha\beta} = 1$ for $\alpha = \beta$ and zero otherwise. We then consider the equation
\begin{equation}\label{eqn: vanishing viscosity equation for time scheme}
v^{\varepsilon}_{i} = v^{\varepsilon}_{i-1} + \left( \mathcal{L}^{\varepsilon}_{i} v^{\varepsilon}_{i} + f_{i}\right) \tau + \sum_{\rho=1}^{d_{1}} \left( \mathcal{M}^{\rho}_{i-1} v^{\varepsilon}_{i-1} + g^{\rho}_{i-1} \right) \xi^{\rho}_{i}
\end{equation}
for each $i \in \{1, \dots, n\}$ with initial condition $v^{\varepsilon}_{0} = v_{0}$. Proving the solvability of \eqref{eqn: vanishing viscosity equation for time scheme} reduces to solving, for each $\omega \in \Omega$, the elliptic problem 
\begin{equation}\label{eqn: elliptic problem}
\mathcal{A}_{i} v^{\varepsilon}_{i} = F_{i}
\end{equation}
for each $i \in \{1, \dots, n\}$ with free term
$$F_{i} := v^{\varepsilon}_{i-1} + \tau f_{i} + \sum_{\rho=1}^{d_{1}} \xi^{\rho}_{i} \left( \mathcal{M}^{\rho}_{i-1} v^{\varepsilon}_{i-1} + g^{\rho}_{i-1}\right)$$ 
and operator  
$$\mathcal{A}_{i} := \left(I - \tau \mathcal{L}^{\varepsilon}_{i}\right)$$ 
where $I$ is the identity. That is, we claim that $\mathcal{A}_{i}$ is
  \begin{enumerate}
  \item bounded, \emph{i.e.}\ $\| \mathcal{A}_{i} \phi \|_{m}^{2} \leq K \| \phi \|_{m+2}^{2}$ for a constant $K$, 
  \item and coercive for sufficiently small $\tau$, \emph{i.e.}\ $\left\langle \mathcal{A}_{i} \phi, \phi \right\rangle \geq \frac{\varepsilon}{2} \| \phi \|_{m+2}^{2}$,
  \end{enumerate}
for all $\phi \in W^{m+2}_{2}$ for every $i \in \{1, \dots, n\}$, where $\left\langle \cdot{},\cdot{} \right\rangle$ denotes the duality pairing between $W^{m+2}_{2}$ and $W^{m}_{2}$ based on the inner product in $W^{m+1}_{2}$. We will obtain the existence of a solution $v^{\varepsilon}_{i}$ to \eqref{eqn: elliptic problem} for each $i \in \{1, \dots, n\}$ via Galerkin approximations (of course, in this instance, one could also use the Lax--Milgram Theorem). 

For integer $p \geq 0$, let $E_{p}$ be the $p$-dimensional subspace of $W^{m+2}_{2}$ spanned by the first $p$ elements of $\{ e_{j}; j \in \mathbf{N}\}$, a collection of vectors from $W^{m+2}_{2}$ forming an orthonormal basis for $W^{m+1}_{2}$. We seek an approximate solution $\phi^{p}_{i} \in E_{p}$ to 
\begin{equation*}
\left\langle \mathcal{A}_{i} \phi^{p}_{i}, e_{k} \right\rangle = \left\langle F_{i}, e_{k} \right\rangle
\end{equation*}
for each $k \in \{1, \dots, p\}$. Rewriting $\phi^{p}_{i} = c^{j}_{p} e_{j}$ for coefficients $c^{j}_{p}$, where the summation convention is used with respect to the repeated index $j \in \{1, \dots, p\}$, we see that $\phi^{p}_{i}$ is an approximate solution if and only if $c^{j}_{p}$ is $W^{m+1}_{2}$-valued and satisfies the system of ordinary differential equations 
\begin{equation*}
c^{j}_{p} \left\langle \mathcal{A}_{i} e_{j}, e_{k} \right\rangle = \left\langle	F_{i}, e_{k}\right\rangle 
\end{equation*}
for each $p$. We derive the estimate 
$$ \frac{\varepsilon}{2} \left\| \phi^{p}_{i} \right\|_{m+2}^{2} \leq \left\langle \mathcal{A}_{i} \phi^{p}_{i}, \phi^{p}_{i} \right\rangle = \left\langle F_{i}, \phi^{p}_{i} \right\rangle \leq \left\| F_{i} \right\|_{m} \left\| \phi^{p}_{i} \right\|_{m+2}.$$
From this we see that $\langle \mathcal{A}_{i} e_{j}, e_{k} \rangle$ is invertible and thus a solution $c^{j}_{p}$, and hence an approximate solution $\phi^{p}_{i}$, exists and moreover we have the estimate $$\mexp \| \phi^{p}_{i} \|_{m+1} \leq 2 \varepsilon^{-1} \mexp \| F_{i} \|_{m}$$ uniformly in $p$. Thus, there exists a $v^{\varepsilon}_{i} \in W^{m+2}_{2}$ and a subsequence $p_{k}$ such that $\phi^{p_{k}}_{i} \to v^{\varepsilon}_{i}$ weakly in $W^{m+2}_{2}$. Therefore, for each $i \in \{1, \dots, n\}$ there exists a $v^{\varepsilon}_{i}$ satisfying \eqref{eqn: elliptic problem} and, moreover, this solution is easily seen to  be unique. Hence, we construct a unique solution to \eqref{eqn: vanishing viscosity equation for time scheme} iteratively. 

Using the existence and uniqueness to the elliptic problem in each interval, we note that 
$$v_{0} + \tau f_{1} + \sum_{\rho=1}^{d_{1}} \left(\mathcal{M}^{\rho}_{0} v_{0} + g^{\rho}_{0}\right)\xi^{\rho}_{1} \in W^{m}_{2},$$ 
by the assumptions on $v_{0}$, $f$, and $g^{\rho}$, for $\rho \in \{1, \dots, d_{1}\}$, and therefore there exists a $v^{\varepsilon}_{1} \in W^{m+2}_{2}$ satisfying 
$$\mathcal{A}_{1} v^{\varepsilon}_{1} = v_{0} + \tau f_{1} + \sum_{\rho=0}^{d_{1}}\left(\mathcal{M}^{\rho}_{0} v_{0} + g^{\rho}_{0}\right)\xi^{\rho}_{1}.$$ 
Further, assuming that there exists a $v^{\varepsilon}_{i} \in W^{m+2}_{2}$ satisfying \eqref{eqn: vanishing viscosity equation for time scheme}, we have that 
$$v^{\varepsilon}_{i} + \tau f_{i+1} + \sum_{\rho=0}^{d_{1}} \left(\mathcal{M}^{\rho}_{i}v^{\varepsilon}_{i} + g^{\rho}_{i}\right)\xi^{\rho}_{i+1} \in W^{m}_{2}$$ 
by the induction hypothesis and Assumption \ref{asm: on initial conditions and free terms}, and therefore there exists a $v^{\varepsilon}_{i+1} \in W^{m+2}_{2}$ satisfying \eqref{eqn: vanishing viscosity equation for time scheme}. Hence, we obtain $v^{\varepsilon}= \left(v^{\varepsilon}_{i}\right)_{i=1}^{n}$ such that each $v^{\varepsilon}_{i} \in W^{m+2}_{2}$ satisfying \eqref{eqn: vanishing viscosity equation for time scheme}.

Finally, we observe that the estimate \eqref{eqn: full estimate for the time scheme in the proof} can be obtained for the solution $v^{\varepsilon}$ to \eqref{eqn: vanishing viscosity equation for time scheme} in a similar manner. In particular, this gives a uniform estimate in $\varepsilon$ for the solution to \eqref{eqn: vanishing viscosity equation for time scheme}. Therefore, there exists a subsequence $\varepsilon_{k} \to 0$ and a $W^{m}_{2}$-valued $\mathcal{F}_{i}$-measurable $v_{i}$ such that $v^{\varepsilon_{k}}_{i}$ converges weakly to $v_{i}$ as $k \to \infty$ in $W^{m}_{2}$ for each $i \in \{1, \dots, n\}$. This $v = \left( v_{i} \right)_{i=1}^{n}$ is the solution to \eqref{eqn: time scheme} and is easily seen to be unique.
\end{proof}

\begin{rmk}
We consider an implicit scheme where the operators $\mathcal{L}$ and $\mathcal{M^{\rho}}$ take values at the points of the time grid. The displacement observed in \eqref{eqn: for the remark} caused by the discretization in time can be avoided if we consider a modified implicit scheme. Namely, if we consider operators defined to be the average over the intervals defined by consecutive points of the time grid, as in \cite{GyongyMillet:2007}, we could then take $\mathcal{M}^{\rho}_{0} v_{0} := 0$. However, we believe such a scheme would be less practical from a computational standpoint.
\end{rmk}

The lemma below is given in \cite{Gyongy:2011} for all $t \in [0,T]$ so, in particular, we have the following for each $i\tau$ for $i \in \{0, \dots, n\}$. This lemma plays the role of Lemma \ref{lem: estimate from Krylov and Rozovskii (1982)} for obtaining the estimate for the space-time scheme.

\begin{lem}\label{lem: Q-estimate}
Let $\phi \in W^{l+2}_{2}$. If Assumptions \ref{asm: parabolicty condition for difference operators} and \ref{asm: coefficients of the difference operators} hold, then for all multiindices $\gamma$ such that $|\gamma| \leq l$, then
\begin{equation*}
Q^{\gamma}_{i}(\phi) = \int_{\mathbf{R}^{d}} 2 \left(D^{\gamma} \phi (x)\right) D^{\gamma} L^{h}_{i}\phi(x) + \sum_{\rho = 1}^{d_{1}} \left| D^{\gamma} M^{h,\rho}_{i} \phi (x) \right|^{2} \mathrm{d}x \leq N \left\| \phi \right\|_{l}^{2}
\end{equation*}
for all $i \in \{1, \dots, n\}$ for a constant $N = N(d,d_{1},d_{2},l,\hat{K}_{0}, \dots, \hat{K}_{l+1},\Lambda)$.
\end{lem}

We are now able to give an estimate for solutions to the space-time scheme that is independent of $h$. Recall that for integer $l \geq 0$, we define the norm 
\begin{equation*}
\left\llbracket \phi \right\rrbracket_{l}^{2} := \mexp \sum_{i=0}^{n} \tau \left\| \phi_{i} \right\|_{l}^{2}
\end{equation*}
and let $\mathbf{W}^{l}_{2}(\tau)$ be the space of $W^{l}_{2}$-valued $\mathcal{F}_{i}$-measurable processes $\phi$ such that $\left\llbracket \phi \right\rrbracket_{l}^{2} < \infty$.

\begin{thm}\label{thm: estimate for space-time scheme}
If Assumptions \ref{asm: on initial conditions and free terms}, \ref{asm: parabolicty condition for difference operators}, and \ref{asm: coefficients of the difference operators} hold, then for each nonzero $h$ equation \eqref{eqn: space-time scheme} admits a unique $W^{l}_{2}$-valued $\mathcal{F}_{i}$-measurable solution $v^{h}$. Moreover, $v^{h}$ satisfies 
\begin{equation}\label{eqn: estimate for space-time scheme}
\mexp \max_{i \leq n} \left\| v^{h}_{i} \right\|_{l}^{2} \leq N \mathcal{K}_{l}^{2}
\end{equation}
for a constant $N = N (d, d_{1}, d_{2}, l,T, \hat{K}_{0}, \dots, \hat{K}_{l+2}, \Lambda)$. If, in addition, $\mathfrak{p}^{\lambda} = \mathfrak{q}^{\lambda} = 0$ for $\lambda \in \Lambda_{0}$, then \eqref{eqn: estimate for space-time scheme} holds for all nonzero $h$.
\end{thm}

\begin{proof}
That \eqref{eqn: space-time scheme} admits a unique $L^{2}$-valued solution follows immediately from the considerations in the proof of Theorem \ref{thm: existence of l-2 valued solution to space-time scheme}. In particular \eqref{eqn: estimate for space-time scheme} can be achieved easily for a constant $N$ depending on $h$, so we see that the solution $v^{h}$ is $W^{l}_{2}$-valued and $\mathcal{F}_{i}$-measurable. To achieve \eqref{eqn: estimate for space-time scheme} for a constant $N$ independent of $h$ (and $\tau$) follows almost immediately from the derivation of the estimate \eqref{eqn: estimate for the time scheme} in the proof of Theorem \ref{thm: solvability of the time scheme with estimate}, using Lemma \ref{lem: Q-estimate} in place of Lemma \ref{lem: estimate from Krylov and Rozovskii (1982)}.

We obtain
\begin{equation}
\label{eqn: intermediate equation with H, I, J}
\left\| v^{h}_{j} \right\|_{l}^{2} \leq \left\| v_{0}\right\|_{l}^{2} + H_{j} + I_{j} + J_{j},
\end{equation}
in the same manner as \eqref{eqn: continuous intermediate equation with H, I, J}, with
\begin{equation*}
H_{j} = 2 \sum_{i=1}^{j} \left(D^{\gamma}v^{h}_{i}, D^{\gamma}\left(L^{h}_{i} v^{h}_{i} + f_{i} \right)\right)\tau,
\end{equation*}
\begin{equation*}
I_{j} = 2 \sum_{i=1}^{j} \sum_{\rho=1}^{d_{1}} \left(D^{\gamma}v^{h}_{i-1}, D^{\gamma}\left(M^{h,\rho}_{i-1} v^{h}_{i-1} + g^{\rho}_{i-1} \right)\right)\xi^{\rho}_{i},
\end{equation*}
and 
\begin{equation*}
J_{j} = \sum_{i=1}^{j} \left\| \sum_{\rho =1}^{d_{1}} D^{\gamma} \left(M^{h,\rho}_{i-1} v^{h}_{i-1} + g^{\rho}_{i-1}\right) \xi^{\rho}_{i} \right\|_{0}^{2}.
\end{equation*}
Then by an application of It{\^o}'s formula, we rewrite $J_{j} = J^{(1)}_{j} + J^{(2)}_{j}$ using
\begin{equation*}
J^{(1)}_{j} = \sum_{i=1}^{j} \left\| \sum_{\rho=1}^{d_{1}} D^{\gamma}\left( M^{h,\rho}_{i-1} v^{h}_{i-1} + g^{\rho}_{i-1} \right)\right\|_{0}^{2} \tau
\end{equation*}
and
\begin{equation*}
J^{(2)}_{j} = \int_{0}^{j\tau} \sum_{\pi,\rho = 1}^{d_{1}} \left(D^{\gamma}\left(M^{h,\pi}_{\kappa(s)} v^{h}_{\kappa(s)} + g^{\pi}_{\kappa(s)}\right), D^{\gamma}\left(M^{h,\rho}_{\kappa(s)} v^{h}_{\kappa(s)} + g^{\rho}_{\kappa(s)}\right)\right) \mathrm{d} Y^{\pi \rho}(s)
\end{equation*}
where $Y^{\pi\rho}(t)$, for $t \in [0,T]$, is defined in the proof of Theorem \eqref{thm: solvability of the time scheme with estimate}.

Now observe that, by Lemma \ref{lem: Q-estimate}, for each $i \in \{0, \dots, n\}$, we have
\begin{align*}
H_{j} + J^{(1)}_{j} 
	&\leq \tau \sum_{\lambda \in \Lambda} \left\| \delta_{h,\lambda} v_{0}\right\|_{l}^{2} + N \tau \sum_{i=1}^{j} \left( Q^{\gamma}_{i} (v^{h}_{i}) + \left(D^{\gamma} v^{h}_{i}, D^{\gamma} f_{i}\right) + \left\|D^{\gamma} g_{i-1}\right\|_{0}^{2}\right) \\
	&\leq N \tau \left\| v_{0}\right\|_{l+1}^{2} + N \tau \sum_{i=1}^{j} \left(\left\|v^{h}_{i}\right\|_{l}^{2} +\left\| f_{i} \right\|_{l}^{2} + \left\| g_{i-1}\right\|_{l}^{2}\right)
\end{align*}
for $h > 0$ where $N = N (l,d,d_{2},\hat{K}_{0}, \dots, \hat{K}_{l+1}, \Lambda)$. Again, here the initial condition $v_{0}$ enters, estimated in the $W^{l+1}_{2}$-norm, due to the displacement caused by the discretization in time when we consider the quadratic form $Q^{\gamma}$ from Lemma \ref{lem: Q-estimate}. If, in addition, $\mathfrak{p}^{\lambda} = \mathfrak{q}^{\lambda} = 0$ for $\lambda \in \Lambda_{0}$, then this last calculation holds for all nonzero $h$. Thus, inequality \eqref{eqn: intermediate equation with H, I, J} becomes
\begin{equation}
\label{eqn: intermediate equation after applying Q bound}
\left\| v^{h}_{j} \right\|_{l}^{2} \leq N \tau \left\|v_{0}\right\|_{l+1}^{2} + N \tau \sum_{i=1}^{j}\left( \left\|v^{h}_{i}\right\|_{l}^{2} + \left\| f_{i} \right\|_{l}^{2} + \left\| g_{i-1}\right\|_{l}^{2}\right) + I_{j} + J^{(2)}_{j}.
\end{equation}
Since $ \mexp I_{j} = 0$ and $\mexp J^{(2)}_{j} = 0$, taking the expectation of \eqref{eqn: intermediate equation after applying Q bound} and taking the sum of $f$ and $g$ over $i \in \{0, \dots, n\}$, we have that
\begin{equation}
\label{eqn: intermediate equation after taking the expectation and maximum}
\mexp \left\| v^{h}_{j} \right\|_{l}^{2}  \leq N \left(\tau \mexp \left\|v_{0}\right\|_{l+1}^{2} +\left\llbracket f \right\rrbracket_{l}^{2} + \left\llbracket g \right\rrbracket_{l}^{2}\right) + N \tau \mexp \sum_{i=1}^{j} \left\| v^{h}_{i} \right\|_{l}^{2}
\end{equation}
for each $j \in \{1, \dots, n\}$. Applying a discrete Gronwall lemma to \eqref{eqn: intermediate equation after taking the expectation and maximum} we have
$$ \mexp \left\| v^{h}_{j}\right\|_{l}^{2} \leq N \left( \tau \mexp \left\| v_{0}\right\|_{l+1}^{2} + \left\llbracket f \right\rrbracket_{l}^{2} + \left\llbracket g \right\rrbracket_{l}^{2}\right) \left(1 - N \tau\right)^{-j}$$
and thus
\begin{equation}
\label{eqn: intermediate bound on max E v after applying discrete Gronwall lemma}
\max_{i \leq n} \mexp \left\| v^{h}_{i}\right\|_{l}^{2} \leq N \left( \tau \mexp \left\| v_{0} \right\|_{l+1}^{2} +  \left\llbracket f \right\rrbracket_{l}^{2} +  \left\llbracket g \right\rrbracket_{l}^{2} \right)
\end{equation}
for a constant $N = N (d,d_{1},d_{2},l,T,\hat{K}_{0}, \dots, \hat{K}_{l+1}, \Lambda)$. In particular, we can use \eqref{eqn: intermediate bound on max E v after applying discrete Gronwall lemma} to eliminate the last term on the right-hand side of \eqref{eqn: intermediate equation after taking the expectation and maximum} by bounding it with terms already appearing on the right-hand side \eqref{eqn: intermediate equation after taking the expectation and maximum}. 

The terms $I$ and $J^{(2)}$ are estimated as in the proof of Theorem \ref{thm: solvability of the time scheme with estimate} with $M^{h,\rho}$ in place of $\mathcal{M}^{\rho}$, using the Burkholder--Davis--Gundy inequality. In particular, we have 
\begin{equation*}
\mexp \max_{i\leq n} \left| I_{i} \right| + \mexp \max_{i \leq n} \left| J^{(2)}_{i} \right| \leq N \left( \tau \mexp \left\| v_{0} \right\|_{l+2} + \left\llbracket f \right\rrbracket_{l+1}^{2} + \left\llbracket g \right\rrbracket_{l+1}^{2} \right)
\end{equation*}
for a constant $N = N(d,d_{1},d_{2},l,T,\hat{K}_{0}, \dots, \hat{K}_{l+2},\Lambda)$. Thus, returning to \eqref{eqn: intermediate equation after applying Q bound} and taking the maximum followed by the expectation we have that 
\begin{equation*}
\mexp \max_{i \leq n} \left\| v^{h}_{i} \right\|_{l}^{2} \leq N \left( \tau \mexp \left\| v_{0} \right\|_{l+2}^{2} + \left\llbracket f \right\rrbracket_{l+1}^{2} + \left\llbracket g \right\rrbracket_{l+1}^{2}\right) = N \mathcal{K}_{l}^{2} < \infty,
\end{equation*}
for a constant $N = N (d,d_{1}, d_{2}, l, T, \hat{K}_{0}, \dots, \hat{K}_{l+2}, \Lambda)$. 
\end{proof}

For the convenience of the reader we record the following lemma, found in \cite{Gyongy:2011}.

\begin{lem}
\label{lem: h derivatives of differences}
Let $\phi \in W_{2}^{p+1}$ and $\psi \in W_{2}^{p+2}$ for an integer $p\geq 0$ and let $\lambda,\mu \in \Lambda_{0}$. Set 
$$ \partial_{\lambda} \phi := \lambda^{j} D_{j} \phi \quad \text{ and } \quad \partial_{\lambda\mu} := \partial_{\lambda}\partial_{\mu}.$$ Then we have
\begin{equation}
\label{eqn: pth derivative in h of a single difference}
\frac{\partial^{p}}{(\partial h)^{p}} \delta_{h,\lambda}\phi (x) = \int_{0}^{1} \theta^{p}\partial_{\lambda}^{p+1} \phi(x+h\theta\lambda) \,\mathrm{d}\theta,
\end{equation}
\begin{equation}\label{eqn: pth derivative in h of a single symmetric difference}
\frac{\partial^{p}}{(\partial h)^{p}} \delta_{\lambda} \phi(x) = \frac{1}{2} \int_{-1}^{1} \theta^{p} \partial^{p+1}_{\lambda} \phi(x+h\theta\lambda) \,\mathrm{d}\theta
\end{equation}
and 
\begin{equation}
\label{eqn: pth derivative in h of a double symmetric difference}
\frac{\partial^{p}}{(\partial h)^{p}} \delta_{\lambda}\delta_{\mu}\psi (x) = \frac{1}{4} \int_{-1}^{1} \int_{-1}^{1} \left(\theta_{1}\partial_{\lambda} - \theta_{2}\partial_{\mu})^{p} \partial_{\lambda\mu} \psi(x + h(\theta_{1}\lambda - \theta_{2}\mu)\right) \,\mathrm{d}\theta_{1}\,\mathrm{d}\theta_{2}
\end{equation}
for almost all $x \in \mathbf{R}^{d}$ for each $h \in \mathbf{R}$. Thus
\begin{equation}\label{eqn: pth derivative of difference of phi at h=0}
\left. \frac{\partial^{p}}{(\partial h)^{p}} \delta_{h,\lambda}\phi \right|_{h=0} = \frac{1}{p+1}\partial^{p+1}_{\lambda} \phi, \quad \left. \frac{\partial^{p}}{(\partial h)^{p}} \delta_{\lambda} \phi \right|_{h=0} = \frac{B_{p}}{p+1}\partial^{p+1}_{\lambda} \phi, 
\end{equation}
and
\begin{equation}\label{eqn: pth derivative of double symmetric difference of phi at h=0}
\left. \frac{\partial^{p}}{(\partial h)^{p}} \delta_{\lambda}\delta_{\mu} \psi \right|_{h=0} = \sum_{r=0}^{p} A_{p,r} \partial^{r+1}_{\lambda} \partial^{p-r+1}_{\mu} \psi,
\end{equation}
where
\begin{equation}\label{eqn: constants B and A}
B_{p} := \begin{cases}0 &\text{if $p$ is odd} \\ 1 & \text{if $p$ is even}\end{cases}, \quad A_{p,r} := \begin{cases} 0 & \text{if $p$ or $r$ is odd} \\ \frac{p!}{(r+1)! (p-r+1)!} & \text{if $p$ and $r$ are even} \end{cases}.
\end{equation}
Furthermore, for integer $l \geq 0$, $\phi \in W_{2}^{p+2+l}$, and $\psi \in W_{2}^{p+3+l}$ one has
\begin{equation}\label{eqn: Taylor expansion error for single difference}
\left\| \delta_{h,\lambda}\phi - \sum_{j=0}^{p}\frac{|h|^{j}}{(j+1)!} \partial_{\lambda}^{j+1} \phi \right\|_{l} \leq \frac{|h|^{p+1}}{(p+2)!} \left\| \partial_{\lambda}^{p+2} \phi \right\|_{l},
\end{equation}
\begin{equation}\label{eqn: Taylor expansion error for symmetric difference}
\left\| \delta_{\lambda}\phi - \sum_{j=0}^{p} \frac{h^{j}}{(j+1)!} B_{j} \partial^{j+1}_{\lambda} \phi \right\|_{l} \leq \frac{|h|^{p+1}}{(p+2)!} \left\| \partial^{p+2}_{\lambda} \phi \right\|_{l},
\end{equation}
and
\begin{equation}\label{eqn: Taylor expansion error for double symmetric difference}
\left\| \delta_{\lambda} \delta_{\mu} \psi - \sum_{j=0}^{p} h^{j} \sum_{r=0}^{j} A_{j,r} \partial^{r+1}_{\lambda} \partial^{j-r+1}_{\mu} \psi \right\|_{l} \leq N |h|^{p+1} \left\| \psi\right\|_{l+p+3},
\end{equation}
where $N = N (|\lambda|, |\mu|, d, p)$.
\end{lem}

For integers $ l \geq 0$ and $r \geq 1$, denote by $W_{h,2}^{l,r}$ the Hilbert space of functions $\phi$ on $\mathbf{R}^{d}$ such that
\begin{equation}
\label{eqn: discrete Sobolev norm}
\left\| \phi \right\|_{l,r,h}^{2} := \sum_{\lambda_{1}, \dots, \lambda_{r} \in \Lambda} \left\| \delta_{h,\lambda_{1}} \times \cdots \times \delta_{h, \lambda_{r}} \phi \right\|_{l}^{2} < \infty
\end{equation}
and set $W_{h,2}^{l,0} := W_{2}^{l}$.

\begin{rmk}
\label{rmk: differences are bounded by derivatives}
Formula \eqref{eqn: pth derivative in h of a single difference} with $p=0$ and Minkowski's integral inequality imply that 
\begin{equation*}
\left\| \delta_{h,\lambda} \phi \right\|_{0} \leq \left\|\partial_{\lambda}\phi \right\|_{0}.
\end{equation*}
By applying this inequality to finite differences of $\phi$ and using induction we can conclude that $W_{2}^{l+r} \subset W_{h,2}^{l,r}$. Further, for any $\phi \in W_{2}^{l+r}$ we have 
\begin{equation*}
\left\| \phi \right\|_{l,r,h} \leq N \left\| \phi \right\|_{l+r},
\end{equation*}
where $N$ depends only on $|\Lambda_{0}|^{2} = \sum_{\lambda \in \Lambda_{0}} |\lambda|^{2}$ and $r$.
\end{rmk}

We now use the preceding observations to obtain estimates in appropriate Sobolev spaces for a system of time discretized equations. Here we use the summation convention with respect to the repeated indices $\lambda,\mu \in \Lambda_{0}$. For $i \in \{0, \dots, n\}$, let
\begin{equation*}
\mathcal{L}^{(0)}_{i} := \mathfrak{a}^{\lambda\mu}_{i} \partial_{\lambda}\partial_{\mu} + \left(\mathfrak{p}^{\lambda}_{i} - \mathfrak{q}^{\lambda}_{i}\right) \partial_{\lambda},
\end{equation*}
\begin{equation*}
\mathcal{M}^{(0)\rho}_{i} := \mathfrak{b}^{\lambda\rho}_{i} \partial_{\lambda},
\end{equation*}
for each $\rho \in \{1,\dots,d_{1}\}$, and for an integer $p\geq 1$ let
\begin{equation*}
\begin{split}
\mathcal{L}^{(p)}_{i} := \sum_{j=0}^{p} A_{p,j} \mathfrak{a}^{\lambda\mu}_{i} \partial^{j+1}_{\lambda} \partial^{p-j+1}_{\mu} + \frac{B_{p}}{p+1} \left(\mathfrak{a}^{\lambda 0}_{i} + \mathfrak{a}^{0\lambda}_{i}\right)\partial^{p+1}_{\lambda} \\ + \frac{1}{p+1}\left(\mathfrak{p}^{\lambda}_{i} + (-1)^{p+1}\mathfrak{q}^{\lambda}_{i}\right) \partial^{p+1}_{\lambda},
\end{split}
\end{equation*}
\begin{equation*}
\mathcal{M}^{(p)\rho}_{i} := \frac{B_{p}}{p+1} \mathfrak{b}^{\lambda \rho}_{i} \partial^{p+1}_{\lambda},
\end{equation*}
\begin{equation*}
\mathcal{O}^{h(p)}_{i} := L^{h}_{i} - \sum_{j=0}^{p} \frac{h^{j}}{j!} \mathcal{L}^{(j)}_{i},
\end{equation*}
and
\begin{equation*}
\mathcal{R}^{h(p)\rho}_{i} := M^{h,\rho}_{i} - \sum_{j=0}^{p} \frac{h^{j}}{j!} \mathcal{M}^{(j)\rho}_{i},
\end{equation*}
where $A_{p,j}$ and $B_{p}$ are defined by \eqref{eqn: constants B and A}. By Assumption \ref{asm: consistency condition}, we have that $\mathcal{L}^{(0)}_{i} = \mathcal{L}_{i}$ and $\mathcal{M}^{(0)\rho}_{i} = \mathcal{M}^{\rho}_{i}$ for all $i \in \{0, \dots, n\}$. For $p \geq 1$, the values of $\mathcal{L}^{(p)}_{i} \phi$ and $\mathcal{M}^{(p)\rho}_{i} \phi$ are obtained by formally taking the $p$th derivatives in $h$ of $L^{h}_{i}\phi$ and $M^{h,\rho}_{i}\phi$, respectively, at $h=0$.

\begin{rmk}
Let $l$ and $p$ be nonnegative integers. For $\phi \in W^{p+2+l}_{2}$ and $\psi \in W^{p+3+l}_{2}$, under Assumptions \ref{asm: coefficients of the differential operators} and \ref{asm: coefficients of the difference operators, less regularity} with $\mathfrak{m} = m$ we have that for $l \leq m$,
\begin{gather}\label{eqn: estimate for O and R}
\left\| \mathcal{O}^{h(p)}_{i}\psi\right\|_{l} \leq N |h|^{p+1} \left\| \psi \right\|_{l+p+3} \intertext{and}
\left\| \mathcal{R}^{h(p)\rho}_{i} \phi\right\|_{l} \leq N |h|^{p+1} \left\| \phi \right\|_{l+p+2}
\end{gather}
for each $\rho \in \{1, \dots, d_{1}\}$ by \eqref{eqn: Taylor expansion error for single difference}--\eqref{eqn: Taylor expansion error for double symmetric difference} for a constant $N = N(p,d,m,\hat{K}_{l}, C_{m}, \Lambda)$.
\end{rmk}

For integer $k \geq 1$, the sequences of random fields $v^{(1)}, \dots, v^{(k)}$ needed in \eqref{eqn: expansion} will be the embedding of random variables taking values in certain Sobolev spaces obtained as solutions to a system of time discretized SPDEs. Namely, as the solution to 
\begin{equation}\label{eqn: system}
\begin{aligned}
\nu^{(p)}_{i} &= \nu^{(p)}_{i-1} + \left( \mathcal{L}_{i} \nu^{(p)}_{i} + \sum_{j=1}^{p} C^{j}_{p} \mathcal{L}^{(j)}_{i} \nu^{(p-j)}_{i} \right) \tau \\ &\qquad \qquad \qquad + \sum_{\rho=1}^{d_{1}}\left(\mathcal{M}^{\rho}_{i-1}\nu^{(p)}_{i-1} + \sum_{j=1}^{p} C^{j}_{p} \mathcal{M}^{(j)\rho}_{i-1}\nu^{(p-j)}_{i-1}\right)\xi^{\rho}_{i},
\end{aligned}
\end{equation}
for $p \in \{1, \dots, k\}$ where $C^{j}_{p} = p(p-1)\dots (p-j+1)/j!$ is the binomial coefficient and $\nu^{(0)}$ is the solution to \eqref{eqn: time scheme} with initial condition $v_{0}$.

\begin{thm}\label{thm: solvability of system with estimate}
If Assumptions \ref{asm: stochastic parabolicity}, \ref{asm: coefficients of the differential operators}, \ref{asm: on initial conditions and free terms}, and \ref{asm: coefficients of the difference operators, less regularity} hold with $\mathfrak{m} = m \geq 3k$, then the system \eqref{eqn: system} admits a unique solution $\nu^{(1)}, \dots, \nu^{(k)}$ for initial condition $\nu^{(1)}_{0} = \dots = \nu^{(k)}_{0} = 0$ such that $\nu^{(p)}$ is $W^{m-3p}_{2}$-valued $\mathcal{F}_{i}$-measurable. Moreover, for each $p \in \{1, \dots, k\}$ if $v^{(p)}$ is a solution, then
\begin{equation}\label{eqn: estimate for solutions to the system}
\mexp \max_{i\leq n} \left\| \nu^{(p)}_{i} \right\|_{m-3p}^{2} \leq N \mathcal{K}_{m}^{2}
\end{equation}
holds for $h >0$ with a constant $N = N(d,d_{1},m,k,T,K_{0},\dots,K_{m+2},C_{m})$. If, in addition, $\mathfrak{p}^{\lambda} = \mathfrak{q}^{\lambda} = 0$ for $\lambda \in \Lambda_{0}$, then \eqref{eqn: estimate for solutions to the system} holds for all nonzero $h$ and $\nu^{(p)} = 0$ for odd $p \leq k$. 
\end{thm}

\begin{proof}
For convenience let 
\begin{equation*}
F^{(p)}_{i} = \sum_{j=1}^{p}C^{j}_{p} \mathcal{L}^{(j)}_{i} \nu^{(p-j)}_{i}
\end{equation*}
and
\begin{equation*}
G^{(p) \rho}_{i} = \sum_{j=1}^{p}C^{j}_{p} \mathcal{M}^{(j)\rho}_{i} \nu^{(p-j)}_{i},
\end{equation*}
where we will write $G^{(p)} = \sum_{\rho=1}^{d_{1}} G^{(p)\rho}$. Observe that for each $p \in \{1, \dots, k\}$ the equation for $\nu^{(p)}$ in \eqref{eqn: system} depends only on $\nu^{(j)}$ for $j < p$ and does not involve any of the unknown processes $\nu^{(j)}$ with indices $j \geq p$. Therefore, we shall prove the solvability of the system and the desired properties on $\nu^{(p)}$ recursively using Theorem \ref{thm: solvability of the time scheme with estimate}.

By Theorem \ref{thm: solvability of the time scheme with estimate}, $\nu^{(0)}$ is $W^{m}_{2}$-valued, $\mathcal{F}_{i}$-measurable, and satisfies \eqref{eqn: estimate for solutions to the system} with $p=0$. Observe that 
$$\left\| \mathcal{L}^{(1)}_{i} \nu^{(0)}_{i} \right\|_{m-2} \leq N \left\| \nu^{(0)}_{i} \right\|_{m}$$ 
for a constant $N = N (m, C_{m})$, recalling the constant $C_{m}$ from Assumption \ref{asm: coefficients of the difference operators, less regularity} in this instance since we only require that $(m-2) \vee 0$ derivatives of the coefficients exist and are bounded, and further that $\mathcal{M}^{(1)}_{i} \nu^{(0)}_{i} = 0$. Therefore, by Theorem \ref{thm: solvability of the time scheme with estimate}, there exists a unique $W^{m-3}_{2}$-valued $\mathcal{F}_{i}$-measurable $v^{(1)}$ satisfying \eqref{eqn: system} with zero initial condition. Moreover, the estimate \eqref{eqn: estimate for solutions to the system} is clearly satisfied in for $p=1$. 

Now we induct on $p$, assuming that for $m \geq 3k \geq 2$ and $p \in \{2, \dots, k\}$ we have unique solutions $\nu^{(1)}$, \dots, $\nu^{(p-1)}$ satisfying the desired properties. Observe that
\begin{equation*}
\left\| \mathcal{L}^{(1)}_{i} \nu^{(p-1)}_{i} \right\|_{m-3p+1} \leq N \left\| \nu^{(p-1)}_{i} \right\|_{m-3(p-1)} \leq N \left\| \nu^{(p-1)}_{i} \right\|_{m-2(p-1)}
\end{equation*}
for a constant $N = N(m,C_{m})$, by Assumption \ref{asm: coefficients of the difference operators, less regularity}, and that $\mathcal{M}^{(1)\rho}_{i} \nu^{(p-1)}_{i} = 0$. Further for $j \geq 2$, if $j$ is even, then
\begin{equation*}
\left\| \mathcal{L}^{(j)}_{i} \nu^{(p-j)}_{i} \right\|_{m-3p+1} \leq N \left\| \mathcal{L}^{(j)}_{i} \nu^{(p-j)}_{i} \right\|_{m-3p+2(j-1)} \leq N \left\| \nu^{(p-j)}_{i} \right\|_{m-3(p-j)}
\end{equation*}
and 
\begin{align*}
\sum_{\rho=1}^{d_{1}} \left\| \mathcal{M}^{(j)\rho}_{i} \nu^{(p-j)}_{i} \right\|_{m-3p+1}^{2} &\leq N \sum_{\rho=1}^{d_{1}} \left\| \mathcal{M}^{(j)\rho}_{i} \nu^{(p-j)}_{i} \right\|_{m-3p+(2j-1)}^{2} \\ &\leq N \left\| \nu^{(p-j)}_{i} \right\|_{m-3(p-j)}^{2}
\end{align*}
or if $j$ is odd, then
\begin{equation*}
\left\| \mathcal{L}^{(j)}_{i} \nu^{(p-j)}_{i} \right\|_{m-3p+1} \leq N \left\| \mathcal{L}^{(j)}_{i} \nu^{(p-j)}_{i} \right\|_{m-3p+2(j-1)} \leq N \left\| \nu^{(p-j)}_{i} \right\|_{m-3(p-j)}
\end{equation*}
and $\mathcal{M}^{(j)\rho}_{i} \nu^{(p-j)}_{i} = 0$, for constants $N = N(m,C_{m})$, by Assumption \ref{asm: coefficients of the difference operators, less regularity}. Therefore, by the induction hypothesis, $F^{(p)}$ is $W^{m-3p+1}_{2}$-valued and $\mathcal{F}_{i}$-measurable, $G^{(p)}$ is $W^{m-3p+1}_{2}$-valued and $\mathcal{F}_{i}$-measurable, and 
\begin{equation}\label{eqn: bound for F-(p) and G-(p)}
\left\llbracket F^{(p)} \right\rrbracket_{m-3p+1}^{2} + \left\llbracket G^{(p)} \right\rrbracket_{m-3p+1}^{2} \leq N \mathcal{K}_{m}^{2}.
\end{equation} 
That is, $F^{(p)} \in \mathbf{W}^{m-3p+1}_{2}(\tau)$ and $G^{(p)} \in \mathbf{W}^{m-3p+1}_{2}(\tau)$. Thus by Theorem \ref{thm: solvability of the time scheme with estimate}, there exists a $W^{m-3p}_{2}$-valued $\mathcal{F}_{i}$-measurable $\nu^{(p)}$ satisfying \eqref{eqn: system} with zero initial condition. Moreover, Theorem \ref{thm: solvability of the time scheme with estimate} yields the estimate 
\begin{equation*}
\mexp \max_{i\leq n} \left\| \nu^{(p)}_{i} \right\|_{m-3p}^{2} \leq N \left( \left\llbracket F^{(p)} \right\rrbracket_{m-3p+1}^{2} + \left\llbracket G^{(p)} \right\rrbracket_{m-3p+1}^{2} \right)
\end{equation*}
for a constant $N = N(d,d_{1},m,k,T,K_{0},\dots,K_{m+1},C_{m})$. Combining this with \eqref{eqn: bound for F-(p) and G-(p)} yields \eqref{eqn: estimate for solutions to the system}. We remark further that the uniqueness of each $\nu^{(p)}$ follows from the uniqueness of the solutions obtained from Theorem \ref{thm: solvability of the time scheme with estimate}.

Note that $\mathcal{M}^{(p)\rho} = 0$ for odd $p \leq k$ by \eqref{eqn: pth derivative of difference of phi at h=0}. Assume, in addition, that $\mathfrak{p}^{\lambda} = \mathfrak{q}^{\lambda} = 0$ for $\lambda \in \Lambda_{0}$. Then also $\mathcal{L}^{(p)} = 0$ for odd $p\leq k$ by \eqref{eqn: pth derivative of difference of phi at h=0} and \eqref{eqn: pth derivative of double symmetric difference of phi at h=0}. Therefore, $F^{(1)}=0$ and $G^{(1)}=0$, which implies $\nu^{(1)} = 0$. Assume that $k \geq 2$ and that for an odd $p \leq k$ we have $\nu^{(j)} = 0$ for all odd $j \leq p$. Then $\mathcal{L}^{(p-j)}\nu^{(j)} = 0$ and $\mathcal{M}^{(p-j)\rho} \nu^{(j)} = 0$ for all $j \in \{1, \dots, p\}$ since either $j$ or $p-j$ is odd. Thus $F^{(p)} = 0$ and $G^{(p)} = 0$. Hence $\nu^{(p)} = 0$ for all odd $p \leq k$. 
\end{proof}

For integers $k \geq 0$ let $\nu^{(1)}, \dots, \nu^{(k)}$ be the solutions to the system \eqref{eqn: system} with zero initial condition coming from Theorem \ref{thm: solvability of system with estimate}. Let 
\begin{equation}\label{eqn: equation for the error function}
\mathfrak{r}^{\tau,h}_{i} := \nu^{h}_{i} - \nu^{(0)}_{i} - \sum_{j=1}^{k} \frac{h^{j}}{j!} \nu^{(j)}_{i}
\end{equation}
for $i \in \{1, \dots, n\}$, where $\nu^{h}$ and $\nu^{(0)}$ are the unique $W^{m}_{2}$-valued solutions to \eqref{eqn: space-time scheme} and \eqref{eqn: time scheme}, respectively, with initial condition $v_{0}$. 

\begin{lem}\label{lem: the error satisfies the equation for the space-time scheme}
Let $\mathfrak{r}^{\tau,h}$ be defined as in \eqref{eqn: equation for the error function}. If Assumptions \ref{asm: stochastic parabolicity}, \ref{asm: coefficients of the differential operators}, \ref{asm: on initial conditions and free terms}, \ref{asm: consistency condition}, \ref{asm: parabolicty condition for difference operators}, \ref{asm: coefficients of the difference operators}, and \ref{asm: coefficients of the difference operators, less regularity} hold with $\mathfrak{m} = m \geq 3k + 4 + l$ for integers $k\geq 0$ and $l \geq 0$, then $\mathfrak{r}^{\tau,h}_{0} = 0$ and $\mathfrak{r}^{\tau,h}$ is $W^{l}_{2}$-valued $\mathcal{F}_{i}$-measurable such that $$\mexp \max_{i \leq n} \left\| \mathfrak{r}^{\tau,h}_{i} \right\|_{l}^{2} < \infty $$ and $\mathfrak{r}^{\tau,h}$ satisfies
\begin{equation}\label{eqn: space-time difference equation for error}
\mathfrak{r}^{\tau,h}_{i} = \mathfrak{r}^{\tau,h}_{i-1} + \left( L^{h}_{i} \mathfrak{r}^{\tau,h}_{i} + \mathfrak{f}^{h}_{i} \right) \tau + \sum_{\rho=1}^{d_{1}} \left(M^{h,\rho}_{i-1} \mathfrak{r}^{\tau,h}_{i-1} + \mathfrak{g}^{h}_{i-1} \right)\xi^{\rho}_{i}
\end{equation}
for $i \in \{1, \dots, n\}$, where 
\begin{equation*}
\mathfrak{f}^{h}_{i} = \sum_{j=0}^{k} \frac{h^{j}}{j!} \mathcal{O}^{h(k-j)}_{i} \nu^{(j)}_{i}
\end{equation*}
and
\begin{equation}
\mathfrak{g}^{h,\rho}_{i-1} = \sum_{j=0}^{k} \frac{h^{j}}{j!} \mathcal{R}^{h(k-j)\rho}_{i-1} \nu^{(j)}_{i-1}.
\end{equation}
Moreover $\mathfrak{f}^{h} \in \mathbf{W}^{l+1}_{2}(\tau)$ and $\mathfrak{g}^{h,\rho} \in \mathbf{W}^{l+1}(\tau)$.
\end{lem}

\begin{proof}
First recall that by Theorem \ref{thm: estimate for space-time scheme}, the solution to the space-time scheme $\nu^{h}$ is $W^{l}_{2}$-valued $\mathcal{F}_{i}$-measurable and satisfies estimate \eqref{eqn: estimate for space-time scheme} owing to Assumptions \ref{asm: on initial conditions and free terms}, \ref{asm: parabolicty condition for difference operators}, and \ref{asm: coefficients of the difference operators} with $m = l$. By Theorem \ref{thm: solvability of the time scheme with estimate}, $\nu^{(0)}$ is $W^{l}_{2}$-valued $\mathcal{F}_{i}$-measurable and satisfies estimate \eqref{eqn: estimate for the time scheme} by Assumptions \ref{asm: stochastic parabolicity}, \ref{asm: coefficients of the differential operators}, and \ref{asm: on initial conditions and free terms} with $m = l$. By Theorem \ref{thm: solvability of system with estimate}, the $\nu^{(j)}$ are $W^{l}_{2}$-valued $\mathcal{F}_{i}$-measurable processes satisfying estimate \eqref{eqn: estimate for solutions to the system} for all $j \in \{1, \dots, k\}$ owing to Assumptions \ref{asm: stochastic parabolicity}, \ref{asm: coefficients of the differential operators}, \ref{asm: on initial conditions and free terms}, and \ref{asm: coefficients of the difference operators, less regularity} with $\mathfrak{m} = m = l$. Thus $\mathfrak{r}^{\tau,h}$ is $W^{l}_{2}$-valued $\mathcal{F}_{i}$-measurable and satisfies
\begin{equation*}
\mexp \max_{i\leq n} \left\| \mathfrak{r}^{\tau,h}_{i} \right\|_{l}^{2} 
< \infty
\end{equation*}
for $\mathfrak{m}= m = l \geq 3k$

One can easily show that $\mathfrak{r}^{\tau,h}$ satisfies \eqref{eqn: space-time difference equation for error} using \eqref{eqn: time scheme}, \eqref{eqn: space-time scheme}, and \eqref{eqn: system} by noting that we can rewrite $\mathfrak{f}^{h}$ and $\mathfrak{g}^{h,\rho}$ as 
\begin{equation*}
\mathfrak{f}^{h} = L^{h} \nu^{(0)} - \mathcal{L}\nu^{(0)} + \sum_{j=1}^{k} \frac{h^{j}}{j!} L^{h} \nu^{(j)} - \sum_{j=1}^{k} \frac{h^{j}}{j!} \mathcal{L} \nu^{(j)} - I
\end{equation*}
and 
\begin{equation*}
\mathfrak{g}^{h,\rho} = M^{h,\rho}\nu^{(0)} - \mathcal{M}^{\rho}\nu^{(0)} + \sum_{j=1}^{k} \frac{h^{j}}{j!} M^{h,\rho}\nu^{(j)} - \sum_{j=1}^{k} \frac{h^{j}}{j!}\mathcal{M}^{\rho} \nu^{(j)} - J
\end{equation*}
for 
\begin{align*}
\sum_{i=0}^{k} \frac{h^{i}}{i!} \sum_{j=0}^{k-i} \frac{h^{j}}{j!} \mathcal{L}^{(j)} \nu^{ (i)} &= \sum_{i=0}^{k-1} \frac{h^{i}}{i!} \sum_{j=1}^{k-i} \frac{h^{j}}{j!} \mathcal{L}^{(j)} \nu^{ (i)}\\
	&= \sum_{i=1}^{k} \sum_{j=0}^{k-i} \frac{h^{i+j}}{i!j!} \mathcal{L}^{(i)} \nu^{ (j)}\\
	&= \sum_{i=1}^{k} \sum_{j=i}^{k} \frac{h^{j}}{i!(j-i)!} \mathcal{L}^{(i)} \nu^{ (j- i)}\\
	&= \sum_{i =1}^{k} \sum_{j=1}^{i} \frac{h^{i}}{j!(i-j)!} \mathcal{L}^{(j)} \nu^{ (i-j)} =: I
\end{align*}
and
\begin{equation*}
\sum_{i=0}^{k} \frac{h^{i}}{i!} \sum_{j=0}^{k-i} \frac{h^{j}}{j!} \mathcal{M}^{(j)\rho} \nu^{ (i)} = \sum_{i=1}^{k}\sum_{j=1}^{i} \frac{h^{i}}{j! (i-j)!} \mathcal{M}^{(j)\rho} \nu^{ (i-j)} =: J
\end{equation*}
where summations over empty sets are taken to be zero.

To prove the last assertion, we note that, by assumption, $m - 3j \geq l + k - j + 4$ for $j \in \{0, 1, \dots, k\}$. Thus by Lemma \ref{lem: h derivatives of differences}, for $j \in \{0, 1, \dots, k\}$, $i \in \{0, \dots, n\}$, and $\omega \in \Omega$,
\begin{equation*}
\left\| \mathcal{O}^{h(k-j)}_{i} \nu^{(j)}_{i} \right\|_{l+1} \leq C \left\| \nu^{(j)}_{i} \right\|_{l+k-j+4} \leq C \left\| \nu^{(j)}_{i} \right\|_{m-3j}
\end{equation*}
and 
\begin{equation*}
\left\| \mathcal{R}^{h(k-j)\rho}_{i} \nu^{(j)}_{i} \right\|_{l+1} \leq C \left\| \nu^{(j)}_{i} \right\|_{l+k-j+3} \leq C \left\| \nu^{(j)}_{i} \right\|_{m-3j}
\end{equation*}
for a constant $C$ independent of the functions and parameters presently under consideration that changes from one instance to the next. Therefore $\mathfrak{f}^{h} \in \mathbf{W}^{l+1}_{2}(\tau)$ and $\mathfrak{g}^{h,\rho} \in \mathbf{W}^{l+1}_{2} (\tau)$.
\end{proof}

In the next section we prove the expansion results.

\section{Proof of Main Results}\label{sec: Proof of Main Results}

We prove a slightly more general result which implies Theorem \ref{thm: generalized expansion} and hence Theorem \ref{thm: expansion}. Here we suppose that $\mathfrak{m} = m$ everywhere.  

\begin{thm}\label{thm: estimate for the error}
Let $\mathfrak{r}^{\tau,h}$ be defined as in \eqref{eqn: equation for the error function}. If Assumption \ref{asm: coefficients of the difference operators} holds with integer $l \geq 0$ and Assumptions \ref{asm: stochastic parabolicity}, \ref{asm: coefficients of the differential operators}, \ref{asm: on initial conditions and free terms}, \ref{asm: consistency condition}, \ref{asm: parabolicty condition for difference operators}, and \ref{asm: coefficients of the difference operators, less regularity} hold with 
\begin{equation}\label{eqn: conditions on m for estimate for the error}
m = 3k + 4 + l
\end{equation}
for integer $k \geq 0$, then, for $h > 0$,
\begin{equation}\label{eqn: estimate for the error}
\mexp \max_{i \leq n} \left\| \mathfrak{r}^{\tau,h}_{i} \right\|_{l}^{2} \leq N \left| h \right|^{2(k+1)}\mathcal{K}_{m}^{2}
\end{equation}
holds with a constant $N = N(d,d_{1},d_{2},m,l, T, K_{0}, \dots, K_{m+2}, \hat{K}_{0}, \dots, \hat{K}_{l+2}, C_{m},\Lambda)$. If, in addition, $\mathfrak{p}^{\lambda} = \mathfrak{q}^{\lambda} = 0$ for $\lambda \in \Lambda_{0}$, then \eqref{eqn: estimate for the error} holds for nonzero $h$ and it suffices to assume $m \geq 3k + 1 +l$ in place of \eqref{eqn: conditions on m for estimate for the error}.
\end{thm}

\begin{proof}
By Lemma \ref{lem: the error satisfies the equation for the space-time scheme}, we have that $\mathfrak{f}^{h} \in \mathbf{W}^{l+1}_{2}(\tau)$ and $\mathfrak{g}^{h,\rho} \in \mathbf{W}^{l+1}_{2}(\tau)$. Thus, by Lemma \ref{lem: the error satisfies the equation for the space-time scheme} and Theorem \ref{thm: estimate for space-time scheme},
\begin{equation}\label{eqn: intermediate estimate for the error}
\mexp \max_{i \leq n} \left\| \mathfrak{r}^{\tau,h}_{i} \right\|_{l}^{2} \leq N \left( \left\llbracket \mathfrak{f}^{h} \right\rrbracket_{l+1}^{2} + \left\llbracket \mathfrak{g}^{h}\right\rrbracket_{l+1}^{2} \right)
\end{equation}
for a constant $N = N(d,d_{1},d_{2},l,T, \hat{K}_{0}, \dots, \hat{K}_{m+2})$. Then by \eqref{eqn: conditions on m for estimate for the error} for $j \in \{0, \dots, k\}$
and by Remark \ref{eqn: estimate for O and R}, we have 
\begin{gather}\label{eqn: intermediate estimate on O and R used for estimate for the error}
\left\| \mathcal{O}^{h(k-j)}_{i} \nu^{(j)}_{i} \right\|_{l+1} \leq N \left| h\right|^{k-j+1} \left\| \nu^{(j)}_{i} \right\|_{l+k-j+4} \leq N \left| h\right|^{k-j+1} \left\| \nu^{(j)}_{i} \right\|_{m-3j}, \intertext{and}
\left\| \mathcal{R}^{h(k-j)\rho}_{i} \nu^{(j)}_{i} \right\|_{l+1} \leq N \left| h\right|^{k-j+1} \left\| \nu^{(j)}_{i} \right\|_{l+k-j+3} \leq N \left| h\right|^{k-j+1} \left\| \nu^{(j)}_{i} \right\|_{m-3j-1}.
\end{gather}
Now using Theorem \ref{thm: solvability of system with estimate}, we see that 
\begin{equation*}
\left\llbracket \mathfrak{f}^{h} \right\rrbracket_{l+1}^{2} + \left\llbracket \mathfrak{g}^{h} \right\rrbracket_{l+1}^{2} \leq N \left| h\right|^{2(k+1)} \mathcal{K}_{m}^{2}
\end{equation*}
which, when taken together with \eqref{eqn: intermediate estimate for the error}, implies \eqref{eqn: estimate for the error}.

If, in addition, $\mathfrak{p}^{\lambda} = \mathfrak{q}^{\lambda} = 0$ for $\lambda \in \Lambda_{0}$, then as in Theorem \ref{thm: solvability of system with estimate} it follows that $\nu^{(j)} = 0$ for all odd $j \leq k$. If $k$ is odd, then clearly $\nu^{(k)} = 0$ and \eqref{eqn: intermediate estimate on O and R used for estimate for the error} holds for $j = k$ and also for $j \geq k-1$ and we need only $m = 3k + 1 + l$. We mention further that $\nu^{(j)} = 0$ for all odd $j \leq k$  in the case $\mathfrak{p}^{\lambda} = \mathfrak{q}^{\lambda} = 0$ for $\lambda \in \Lambda_{0}$ also follows from \eqref{eqn: estimate for the error}, now valid for all nonzero $h$, since $\nu^{h}$ and $\nu^{-h}$ are the $L^{2}$-valued solutions to \eqref{eqn: space-time scheme} with initial condition $v_{0}$ and we must have $\nu^{h} = \nu^{-h}$ due to the uniqueness of solutions.
\end{proof}

We have the following corollary to Theorem \ref{thm: estimate for the error} which implies Theorem \ref{thm: generalized expansion} and hence Theorem \ref{thm: expansion}. Let $R^{\tau,h} := \mathcal{I} \mathfrak{r}^{\tau,h}$, where $\mathcal{I}$ is the embedding operator from $\ref{lem: embedding}$, and notice that $R^{\tau,h}_{i} \in \ell^{2}(\mathbf{G}_{h})$ for all $i \in \{0, \dots, n\}$.

\begin{cor}\label{cor: estimate for difference of expansion error}
If in Theorem \ref{thm: estimate for the error} we have $l > p + d/2$ for integer $p \geq 0$, then, for $\lambda \in \Lambda^{p}$ and $h > 0$,
\begin{equation*}
\mexp \max_{i \leq n} \sup_{x \in \mathbf{R}^{d}} \left| \delta_{h,\lambda} R^{\tau,h}_{i}(x) \right|^{2} \leq N h^{2(k+1)} \mathcal{K}_{m}^{2}
\end{equation*}
holds with a constant $N = N(d, d_{1},d_{2}, m, l, T, K_{0}, \dots, K_{m+2}, \hat{K}_{0}, \dots, \hat{K}_{l+2}, C_{m},\Lambda)$. 
\end{cor}

\begin{proof}
By Sobolev's embedding of $W^{l-p}_{2}$ into $\mathcal{C}_{b}$ and Remark \ref{rmk: differences are bounded by derivatives} we have that 
\begin{equation*}
\mexp \max_{i\leq n} \sup_{x \in \mathbf{R}^{d}} \left| \delta_{h,\lambda} R^{\tau,h}_{i}(x) \right|^{2} \leq N \mexp \max_{i \leq n} \left\| \mathfrak{r}^{\tau,h}_{i} \right\|_{l-p, p, h}^{2}
	\leq N \mexp \max_{i \leq n} \left\| \mathfrak{r}^{\tau,h}_{i} \right\|_{l}^{2}
\end{equation*}
which, together with Theorem \ref{thm: estimate for the error}, yields the desired result.
\end{proof}

It only remains to explain how the corollary implies Theorem \ref{thm: generalized expansion} and hence Theorem \ref{thm: expansion}. Using Sobolev's embedding, we find continuous versions of the $L^{2}$-valued solutions to the space-time scheme, the time scheme, and the system of time discretized equations. Then we notice that the restriction of the $L^{2}$-valued solution to \eqref{eqn: space-time scheme} onto $\mathbf{G}_{h}$ agrees with the unique $\ell^{2}(\mathbf{G}_{h})$ valued solution. This argument can be found in \cite{Hall:2012}, nevertheless we reproduce it here for the convenience of the reader. 

For $\mathcal{I}: W^{l}_{2} \to \mathcal{C}_{b}$ from Lemma \ref{lem: embedding}, Theorem \ref{thm: generalized expansion} follows by considering the embeddings $\hat{v}^{h} := \mathcal{I} \nu^{h}$, where $\nu^{h}$ is the unique $L^{2}$-valued solution to \eqref{eqn: space-time scheme} with initial condition $v_{0}$, and $v^{(j)} := \mathcal{I} \nu^{(j)}$, for $j \in \{ 0, \dots, k\}$, where $\nu^{(0)}$ is the unique $L^{2}$-valued solution to \eqref{eqn: time scheme} with initial condition $v_{0}$ and the processes $\nu^{(1)}$, \dots, $\nu^{(k)}$ are the solutions to the system of time discretized SPDEs \eqref{eqn: system} as given in Theorem \ref{thm: solvability of system with estimate}. By Theorem \ref{thm: estimate for space-time scheme}, $\nu^{h}$ is $W^{l}_{2}$-valued and $\mathcal{F}_{i}$-measurable for all $i \in \{1, \dots, n\}$. For each $j \in \{ 1, \dots, k\}$, the $\nu^{(j)}$ are $W^{m-3k}$-valued processes by Theorem \ref{thm: solvability of the time scheme with estimate}. Since $l > d/2$ and $m-3k > d/2$, the processes $\hat{v}^{h}$ and $v^{(j)}$ are well defined and clearly \eqref{eqn: equation for the error function} implies \eqref{eqn: expansion} with $\hat{v}^{h}$ in place of $v^{h}$. That is, we have the expansion for a continuous version of the $L^{2}$-valued solution.

Next we show that the restriction of the $L^{2}$-valued solution to the grid $\mathbf{G}_{h}$, a set of Lebesgue measure zero,  is indeed equal almost surely to the unique $\ell^{2}(\mathbf{G}_{h})$-valued solution that one would naturally obtain from \eqref{eqn: space-time scheme}. That is, we show that
\begin{equation}
\label{eqn: equality of the restriction} 
\hat{v}^{h}_{i} (x) = v^{h}_{i}(x)
\end{equation} 
almost surely for all $i \in \{1, \dots, n\}$ and for all $x \in \mathbf{G}_{h}$ where $v^{h}$ is the unique $\mathcal{F}_{i}$-measurable $\ell^{2}(\mathbf{G}_{h})$-valued solution of \eqref{eqn: space-time scheme} from  Theorem \ref{thm: estimate for space-time scheme}. Therefore, for a compactly supported nonnegative smooth function $\phi$ on $\mathbf{R}^{d}$ with unit integral and for a fixed $x \in \mathbf{G}_{h}$ we define $$\phi_{\varepsilon}(y) := \phi\left( \frac{y-x}{\varepsilon} \right)$$ for $y \in \mathbf{R}^{d}$ and $\varepsilon > 0$. Recall, by Remark \ref{rmk: on Sobolev's embedding}, that we can obtain versions of $v_{0}$, $f$, and $g^{\rho}$, for $\rho \in \{1, \dots, d_{1}\}$, that are continuous in $x$. Since $\hat{v}^{h}$ is a $L^{2}$-valued solution of \eqref{eqn: space-time scheme} for each $\varepsilon$, almost surely
\begin{equation*}
\begin{split}
\int_{\mathbf{R}^{d}} \hat{v}^{h}_{i}(y) \phi_{\varepsilon}(y) \, \mathrm{d}y = \int_{\mathbf{R}^{d}} \hat{v}^{h}_{i-1}(y) \phi_{\varepsilon}(y) \, \mathrm{d}y + \tau \int_{\mathbf{R}^{d}} (L^{h}_{i}\hat{v}^{h}_{i} + f_{i})(y) \phi_{\varepsilon}(y) \, \mathrm{d}y\\
	 + \sum_{\rho=1}^{d_{1}}\xi^{\rho}_{i} \int_{\mathbf{R}^{d}} ( M^{h,\rho}_{i-1} \hat{v}^{h}_{i-1} + g^{\rho}_{i-1})(y) \phi_{\varepsilon}(y) \, \mathrm{d}y
\end{split}
\end{equation*}
for each $i \in \{1, \dots, n\}$. Letting $\varepsilon \to 0$, we see that both sides converge for all $i \in \{1, \dots, n\}$ and $\omega \in \Omega$. Therefore, almost surely 
\begin{equation*}
\hat{v}^{h}_{i}(x) = \hat{v}^{h}_{i-1}(x) + \left(L^{h}_{i}\hat{v}^{h}_{i}(x) + f_{i}(x)\right) \tau + \sum_{\rho=1}^{d_{1}}\left( M^{h,\rho}_{i-1} \hat{v}^{h}_{i-1}(x) + g^{\rho}_{i-1}(x) \right)\xi^{\rho}_{i}
\end{equation*}
for all $i \in \{1, \dots, n\}$. Moreover, by Lemma \ref{lem: embedding}, the restriction of $\hat{v}^{h}$, the continuous version of $\nu^{h}$, onto $\mathbf{G}_{h}$ is an $\ell^{2}(\mathbf{G}_{h})$-valued process. Hence, \eqref{eqn: equality of the restriction} holds, due to the uniqueness of the $\mathcal{F}_{i}$-measurable $\ell^{2}(\mathbf{G}_{h})$-valued solution of \eqref{eqn: space-time scheme} for any $\mathcal{F}_{0}$-measurable $\ell^{2}(\mathbf{G}_{h})$-valued initial data. This finishes the proof of Theorems \ref{thm: generalized expansion} and  \ref{thm: expansion}.

\section{Acknowledgements}\label{sec: Acknowledgements}
The author would like to express gratitude towards his supervisor, Professor Istv{\'a}n Gy{\"o}ngy, for the guidance and encouragement offered during the preparation of these results, which  form part of the author's Ph.D.\ thesis.



\end{document}